\RequirePackage{fix-cm}
\documentclass[smallextended]{svjour3}       % onecolumn (second format)
\smartqed  % flush right qed marks, e.g. at end of proof
\usepackage[caption=false,font=footnotesize]{subfig}
\usepackage{graphicx}
\usepackage{color}
\usepackage{rotating, booktabs}
\usepackage{algpseudocode}
\usepackage[ruled]{algorithm}

\usepackage{pgfplots}
\pgfplotsset{compat=1.13}

\usepackage{graphicx}
\usepackage{amsmath,amstext,amsbsy,amssymb}
\usepackage[cdot,amssymb]{SIunits}
\usepackage[latin1]{inputenc}
\usepackage{float}
\usepackage{comment}

\def\MM{{\mathcal{M}}}

\def\AA{{\Lambda}}
\def\CC{{\mathcal{C}}}

\def\HH{{\mathcal{H}}}
\def\MM{{\mathcal{M}}}

\def\LL{{\mathcal{L}}}

\def\CCiso{{\mathcal{C}_{I}}}
\def\CCtra{{\mathcal{C}_{TI}}}

\def\SS{{\mathcal{S}}}

\def\EE{{\mathcal{E}}}
\def\FF{{\mathcal{F}}}

\def\bQ{{\mathbf{Q}}}
\def\bY{{\mathbf{Y}}}
\def\bG{{\mathbf{G}}}
\def\bW{{\mathbf{W}}}
\def\bX{{\mathbf{X}}}
\def\bS{{\mathbf{S}}}
\def\bD{{\mathbf{D}}}
\def\bRd{{\mathbf{R}}_d}
\def\bRc{{\mathbf{R}}_c}
\def\bI{{\mathbf{I}}}
\def\bL{{\mathbf{L}}}

\def\name{IPM-Proj }
\def\namedot{IPM-Proj}

\newcommand{\req}[1]{(\ref{#1})}

\newcommand{\Commentcp}[1]{\textcolor{blue}{Cristina: #1}}

\begin{document}
%
% \title{Insert your title here%\thanks{Grants or other notes
% %about the article that should go on the front page should be
% %placed here. General acknowledgments should be placed at the end of the article.}
% }

%\title{Projections onto the cone of negative semidefinite symmetric tensors}
\title{A semidefinite programming approach for the projection onto the cone of negative semidefinite symmetric tensors with applications to solid mechanics}


\titlerunning{A SDP approach for the projection onto $Sym^{-}$}        % if too long for running head

\author{Cristina Padovani \and    Margherita Porcelli
}


\institute{C. Padovani \at
              Institute of Information Science and Technologies ``A. Faedo'', ISTI--CNR, Via Moruzzi 1 Pisa, Italy\\
              \email{cristina.padovani@isti.cnr.it}           %  \\
           \and
           M. Porcelli \at
             Dipartimento  di Matematica, AM$^2$, Universit\`a di Bologna,
           Piazza di Porta San Donato 5, 40126 Bologna, Italia.
         \at  Institute of Information Science and Technologies ``A. Faedo'', ISTI--CNR, Via Moruzzi 1 Pisa, Italy\\
           \email{margherita.porcelli@unibo.it}
}

\date{Received: date / Accepted: date}

\maketitle

\begin{abstract}
We propose an algorithm for computing the projection of a symmetric second-order tensor onto the cone of negative semidefinite symmetric tensors with respect to the inner product defined by an assigned positive definite symmetric fourth-order tensor $\CC$.  The projection problem is written as a semidefinite programming problem and an algorithm based on a primal-dual path-following interior point method coupled with a  Mehrotra's predictor-corrector approach is proposed. Implementations based on
direct methods are theoretically and numerically investigated taking into account tensors $\CC$ arising in the modelling of masonry-like materials. 

%In particular, we implement two of the most used symmetrization schemes, that is the NT and the AHO direction focusing on the solution of the
%linear system by direct methods. The two approach are then compared, analysing the conditioning of the linear systems deriving from the NT and AHO schemes,  and validated on some application oriented numerical tests.


\end{abstract}

\section{Introduction}
\label{sec1} \vspace{-2pt}
Matrix nearness problems are introduced in \cite{Higham1} where for a fixed matrix $A$, the problem of finding the nearest member of some given class of matrices is addressed, where distance is measured in a
matrix norm. The problem of approximating a matrix with a positive semidefinite symmetric matrix is ubiquitous in scientific computing, see e.g. \cite{H_M_2012,Higham1,Higham3,Higham2}. In particular, in \cite{Higham3}, numerical methods are proposed to calculate the minimum distance between a matrix $A$ and a positive definite symmetric matrix $X$, considering the Frobenius norm and the 2-norm. 
Motivated by relevant applications in finance industry, more recent contributions \cite{Anjos,H_M_2012,Higham2,Malick,QiSun}  deal with the computation of the projection of a symmetric matrix onto the set of correlation matrices, namely positive semidefinite symmetric matrices with ones on the diagonal.

In this paper, we are concerned with the computation of the projection of a symmetric second-order tensor onto the cone of negative semidefinite symmetric tensors with respect to the inner product defined by an assigned positive definite symmetric fourth-order tensor $\CC$. In particular, for a given symmetric tensor $\mathbf{D}$, we want to minimize the distance $\phi(\mathbf{Y})=\parallel {\mathbf{D}-\mathbf{Y}}\parallel_{\CC}^{2}$ between  $\mathbf{D}$ and $\mathbf{Y}$, with $\mathbf{Y}$ belonging to the cone $Sym^{-}$ of negative semidefinite symmetric tensors.
Problems similar to the minimization of $\phi(\mathbf{Y})$ in $Sym^{-}$ have been addressed in \cite{H_M_2012}, where numerical methods for conic projection problems are presented.
In particular, in \cite{H_M_2012}, the problem of minimizing the standard Frobenius distance between a given matrix $C$ and a symmetric positive semidefinite matrix $X$ and satisfying a further equality linear constraint is introduced. Then, in \cite{H_M_2012} (equation (7)) the focus is on the more general problem of finding the projection of a vector $c$ onto the intersection of a convex closed set and a convex polyhedron defined by affine inequalities,
with respect to
$\parallel x\parallel_{Q}^{2}=x\cdot Qx$, the norm associated to a positive definite matrix $Q$. In particular, equation (7) is
the vector counterpart of our tensor problem in the cone $Sym^{-}$.

The relevance of the projection problem proposed in the present paper is twofold. Firstly, it is a generalization of the problem dealt with in \cite{Higham3}, as in the place of the Frobenius norm, we consider a norm induced by a scalar product defined by an assigned fourth-order tensor $\CC$.
Secondly, its solution allows to model masonry-like materials \cite{GDP,SPRING}.
In fact, the stress tensor for materials that do not withstand tension can be obtained by suitably projecting the strain tensor onto the cone of negative semidefinite symmetric tensors. Solving such a projection problem has a crucial role in solid mechanics and civil engineering applications, as it allows to calculate a solution to the equilibrium problem of masonry constructions. Fourth-order tensor $\CC$ contains the mechanical properties of the masonry material and can take different forms depending on the anisotropy of the material\cite{SPRING}. Apart from the isotropic case, for which the explicit solution to the projection problem is available, in the general anisotropic case numerical methods are necessary to calculate the approximate solution.
Rather than designing efficient algorithms for large scale problems as done in most of the literature \cite{Anjos,H_M_2012,Malick}, the focus of this work is providing an accurate and cost-effective numerical procedure for small size projection, as done in \cite{HN}, where an algorithm to compute the polar decomposition of a $3\times 3$ matrix is proposed. This framework is 
strongly motivated by the application in the field of solids mechanics and, in particular, on masonry-like materials, where the dimension of the addressed problem is very low. Indeed the considered second-order tensors are linear functions from a three-dimensional vector space into itself.
Moreover, the accurate solution of a projection problem is required for each Gauss points of each element constituting the finite element discretization of the masonry structure under examination. Thus, similarly to \cite{HN}, our algorithm must solve a large number of small size problems accurately. 

Inspired by the works \cite{Anjos,QSDP2,QSDP1} for the large-scale setting, we propose to
solve the projection problem by using a semidefinite programming (SDP) approach and
developing an interior point algorithm that exploits the peculiarities of the problem
under consideration.  Interior point methods stand out as reliable algorithms which enjoy enviable convergence properties and usually provide accurate solutions within reasonable time. Several proposals are available in the literature for both general and
application oriented SDPs, see e.g. \cite{bgp1,bgp2,kocvara,SDLCP1,ToddSurvey,QSDP1} and references therein. 

In this work we first show that our projection problem can be reformulated as a special monotone semidefinite linear complementarity problem (SDLCP) and observe that it is equivalent to a convex Quadratic SemiDefinteProgramming (QSDP) problem where there are no linear equality constraints. Then we describe a primal-dual path-following interior point that uses Mehrotra's predictor-corrector steps \cite{QSDP2,QSDP1,SDTP3-guide} and adapted it to our QSDP.
In particular, we considered two of the most used symmetrization schemes, that is
the Nesterov-Todd (NT) direction and the Alizadeh-Haeberly-Overton (AHO) one \cite{ToddDir}, and focused on the solution of the
linear system by direct methods. As a major contribution of this work we show that, when a very accurate solution is required, the use of the popular  NT direction yields  highly ill-conditioned Schur complement linear systems that may prevent the computation of an accurate solution. On the other hand, we provide
a formulation of the Newton's equation with a much favourable condition
number when the AHO direction is used and the $\CC$ has a special form of  interest in solid mechanics. For this case, a theoretical insight of this behaviour is given.
The addressed theoretical issues are validated on a number of application oriented numerical tests. 

As an outcome of the obtained good numerical results, the proposed algorithm will be implemented in the finite element code NOSA-ITACA \cite{nosaitaca} developed at ISTI-CNR for the structural analysis of masonry constructions. The implementation and the application of the code to a case study  will be the subject of future work.

This paper is organized as follows. In Section \ref{roba}, we list several notions and definitions to be used in the paper that attempt to merge standard notation used in solid mechanics and in SDPs. Section \ref{sec2} describes the projection problem in the space of symmetric tensors equipped with the scalar product associated with a positive definite symmetric fourth-order tensor $\CC$. Some results deriving from the minimum norm theorem are proved, including the possibility of expressing the projection problem as a complementarity problem. In Section \ref{Special}, some special forms of $\CC$ of interest in solid mechanics are presented, focusing on isotropic and transversely isotropic $\CC$. In particular, the explicit expression of the projection for isotropic $\CC$ is provided. The transversely isotropic case can not be solved explicitly, and the projection is calculated only for a restricted class of tensors $\mathbf{D}$. 
Section \ref{numeric} contains the description of the   primal-dual path-following interior point algorithm adopted for the efficient and accurate solution of the complementarity problem associated with the projection problem.
Section \ref{test} is devoted to the description of the numerical experience.
First the implementation of the proposed algorithm is described. Then the data sets are introduced and numerical results are discussed. Conclusions are drawn in Section \ref{sec:end}.

\section{Notations and preliminaries}\label{roba}
Let $\mathcal{V}$  be a real vector space of dimension 3 with the inner product $\cdot $. Let $Lin$ be the set of all second-order tensors (a second-order tensor is a linear application from $\mathcal{V}$ in itself) with the inner product $%
\mathbf{A}\bullet \mathbf{B}=tr(\mathbf{A}^{T}\mathbf{B})$ for any $\mathbf{A},%
\mathbf{B}\in $ $Lin$, with $\mathbf{A}^{T}$ the transpose of
$\mathbf{A}$ and let $\parallel \mathbf{A} \parallel = \sqrt{\mathbf{A} \bullet \mathbf{A}}$ be the associated Frobenius norm.

For $Sym$ the subspace of symmetric tensors, $Sym^{-}$,
$Sym^{+}$ and $Sym^{++}$ are the
sets of all negative semidefinite, positive semidefinite  and positive definite elements of $Sym$, respectively. $Orth$ denotes the group of all orthogonal tensors.

Given the tensors $\mathbf{A}$ and $\mathbf{B}$, we denote by $%
\mathbf{A}\otimes \mathbf{B}$ the fourth-order tensor (a fourth-order tensor is a linear application from $Lin$ in itself) defined by
$$\mathbf{A} \otimes \mathbf{B}(\mathbf{H})=(\mathbf{B}\bullet \mathbf{H})\mathbf{A}
 \,\,\,\textup{for all}\,\,\,\mathbf{H}\in Lin, $$
by  $\mathbf{A}\odot \mathbf{B}$ the fourth-order tensor defined by
\begin{equation*}
\mathbf{A}\odot \mathbf{B}(\mathbf{H})=\frac{1}{2}(\mathbf{B} \mathbf{H}\mathbf{A}^T+\mathbf{A} \mathbf{H}^T\mathbf{B}^T), \,\,\,\textup{for all}\,\,\,\mathbf{H}\in Lin, \label{symK}
\end{equation*}
and by $\mathbb{I}_{Sym}$ the fourth-order
identity tensor on $Sym$.
For $\mathbf{a}$ and $\mathbf{b}$ vectors, the dyad $\mathbf{a}%
\otimes \mathbf{b}$ is defined by $\mathbf{a}\otimes \mathbf{bh}=(\mathbf{b}%
\cdot \mathbf{h})\mathbf{a},$ for any vector $\mathbf{h}$.

Let $\CC$ be a fourth-order tensor from $Sym$ to $Sym$.
Let us assume that  $\CC$ is symmetric, i.e.
\begin{equation}
\mathbf{A}\bullet \CC(\mathbf{B})=\mathbf{B}\bullet
\CC(\mathbf{A}),\,\,\,\,\textup{for all}\,\,\,\,
\mathbf{A},\mathbf{B}\in Sym, \label{symmc}
\end{equation}
\noindent and positive definite on $Sym$, i.e.
\begin{equation}
\mathbf{A}\bullet \CC(\mathbf{A})> 0\,\,\,\,\textup{for
all}\,\,\,\, \mathbf{A}\in Sym,\,\,\,\, \mathbf{A}\neq 0.
\label{posc}
\end{equation}
Because of (\ref{posc}) and (\ref{symmc}) $\CC$ is invertible and its inverse $\CC^{-1}$ is symmetric and positive definite. Moreover, properties (\ref{posc}) and (\ref{symmc}) allow defining the following inner product  $\circ $ on $Sym$,
\begin{equation}
\mathbf{A}\circ \mathbf{B}=\mathbf{A}\bullet \CC(\mathbf{B}),\,\,\,\,\textup{for
}\,\,\,
\mathbf{A},\mathbf{B}\in Sym,
\label{ipC}
\end{equation}
and the associated $\CC$-norm
\begin{equation*}
\parallel \mathbf{A}\parallel_{\CC}^{2}=\mathbf{A}\circ \mathbf{A}.
\label{Cnorm}
\end{equation*}
Let $\mathsf{P}= (\mathbf{p}_{1},\mathbf{p}_{2},\mathbf{p}_{3})$ be an orthonormal basis of  $\mathcal{V}$. For $\mathbf{D}\in Sym$ and $\CC$ symmetric and positive definite, the components $D_{ij}$  of $\mathbf{D}$, $\CC_{ijkl}$ of $\CC$ and $\CC_{ijkl}^{-1}$ of $\CC^{-1}$ with respect to $\mathsf{P}$ are
\begin{equation*}
D_{ij}=\mathbf{p}_{i} \cdot \mathbf{D} \mathbf{p}_{j},  \textup{ \ } \textup{with} \textup{ \ } D_{ij}=D_{ji}, \,\,\, i,j=1,2,3
\label{Eij}
\end{equation*}
\begin{equation*}
\CC_{ijkl}=\mathbf{p}_{i}  \cdot \CC(\frac{\mathbf{p}_{k} \otimes  \mathbf{p}_{l}+\mathbf{p}_{l} \otimes  \mathbf{p}_{k}} {2}) \mathbf{p}_{j}
\label{Cijkl},  \textup{ \ \ }  i,j,k,l=1,2,3,
\end{equation*}
\begin{equation*}
\CC_{ijkl}^{-1}=\mathbf{p}_{i} \cdot \CC^{-1}(\frac{\mathbf{p}_{k} \otimes  \mathbf{p}_{l}+\mathbf{p}_{l} \otimes  \mathbf{p}_{k}} {2}) \mathbf{p}_{j}
\label{C-1ijkl},  \textup{ \ \ }  i,j,k,l=1,2,3.
\end{equation*}
These components are reported in the Appendix for the special forms of $\CC$ described in Section \ref{Special}.

Because $\CC$ and $\CC^{-1}$ are symmetric fourth-order tensors from $Sym$ to $Sym$, their components satisfy the following equalities
\begin{equation}
\CC_{ijkl}=\CC_{klij}, \textup{ \ \ } \CC_{ijkl}=\CC_{jikl}=\CC_{ijlk}, \,\,\, i,j,k,l=1,2,3,
\label{simmetrieC}
\end{equation}
\begin{equation}
\CC_{ijkl}^{-1}=\CC_{klij}^{-1}, \textup{ \ \ } \CC_{ijkl}^{-1}=\CC_{jikl}^{-1}=\CC_{ijlk}^{-1}, \,\,\,\, i,j,k,l=1,2,3.
\label{simmetrieC-1}
\end{equation}
With these notations, for a given symmetric tensor $\mathbf{A}$, the symmetric tensor
\begin{equation*}
\mathbf{B}=\CC(\mathbf{A})
\label{B=CA}
\end{equation*}
has components
\begin{equation}
B_{ij}=\sum_{kl} \CC_{ijkl}A_{kl}=\sum_{k} \CC_{ijkk}A_{kk}+2 \sum_{k<l} \CC_{ijkl}A_{kl}.
\label{funct4}
\end{equation}
It may be convenient to adopt a vector notation in the place of the tensor notation described above. Thus, a symmetric tensor
$\mathbf{A}$ is replaced by the vector $\mathbf{a}$ with the six components
\begin{equation*}
(A_{11}, \sqrt{2}A_{12}, A_{22}, \sqrt{2}A_{13}, \sqrt{2}A_{23}, A_{33})^T,
\label{ab}
\end{equation*}
such that $\mathbf{a}\cdot \mathbf{a}=\mathbf{A}\bullet \mathbf{A}=tr(\mathbf{A}^2)$. Then, for $\mathbf{b}$ the vector associated to $\mathbf{B}$, from
(\ref{funct4}) we get
\begin{equation*}
\mathbf{b}=\mathbf{\widetilde C}\mathbf{a},
\label{b=Ca}
\end{equation*}
where the matrix of the components of $\mathbf{\widetilde C}$ is
\begin{equation*}
\begin{pmatrix}
\CC_{1111}&\sqrt{2}\CC_{1112}&\CC_{1122}&\sqrt{2}\CC_{1113}&\sqrt{2}\CC_{1123}&\CC_{1133}\\
\sqrt{2}\CC_{1211}&2\CC_{1212}&\sqrt{2}\CC_{1222}&2\CC_{1213}&2\CC_{1223}&\sqrt{2}\CC_{1233}\\
\CC_{2211}&\sqrt{2}\CC_{2212}&\CC_{2222}&\sqrt{2}\CC_{2213}&\sqrt{2}\CC_{2223}&\CC_{2233}\\
\sqrt{2}\CC_{1311}&2\CC_{1312}&\sqrt{2}\CC_{1322}&2\CC_{1313}&2\CC_{1323}&\sqrt{2}\CC_{1333}\\
\sqrt{2}\CC_{2311}&2\CC_{2312}&\sqrt{2}\CC_{2322}&2\CC_{2313}&2\CC_{2323}&\sqrt{2}\CC_{2333}\\
\CC_{3311}&\sqrt{2}\CC_{3312}&\CC_{3322}&\sqrt{2}\CC_{3313}&\sqrt{2}\CC_{3323}&\CC_{3333}\\
\end{pmatrix}.
\end{equation*}

Finally, we denote by $\lambda_{\min}(\mathbf{A})$,  and $\lambda_{\max}(\mathbf{A})$ the minimum and maximum eigenvalue of a tensor $\mathbf{A}$, respectively. Analogous notation is adopted for a fourth-order tensor $\mathcal{A}$.
%%%%%%%%%%%%%%%%%%%%%%%%%%%%%%%%%%%%%%%%%%%%%%%%%%%%%%%%%%%%%%%%%%%%%%%%%%%%%%%%%%%%%%%%%%%%%
\section{The projection problem}
\label{sec2}
Given $\mathbf{D}\in Sym$, we address the problem of minimizing the following functional
\begin{equation}
 \phi(\mathbf{Y})=\parallel {\mathbf{D}-\mathbf{Y}}\parallel_{\CC}^{2}=%
(\mathbf{D}-\mathbf{Y}) \bullet\CC(\mathbf{D}-\mathbf{Y}),
\label{functional}
\end{equation}
over the set of negative semidefinite symmetric tensors $Sym^{-}$.
Since $Sym^{-}$ is a closed convex  cone of $Sym$, in view of the minimum norm theorem \cite{FA}, there exists a unique minimum point
$\bY^*\in Sym^{-}$ for the functional (\ref{functional}). Moreover, $\bY^*$ is the minimum point of (\ref{functional}) if and only if it satisfies the variational inequality
\begin{equation*}
(\mathbf{D}-\bY^*)\circ(\mathbf{Y}-\bY^*)\leq0,\textup{ \ \ }\forall\ \mathbf{Y}\in Sym^{-},
\label{varineq}
\end{equation*}
which, expressed in terms of the inner product $\bullet$, reads
\begin{equation}
\CC(\mathbf{D}-\bY^*)\bullet(\mathbf{Y}-\bY^*)\leq0,\textup{ \ \ }\forall\ \mathbf{Y}\in Sym^{-}.
\label{varineq2}
\end{equation}

The following proposition gives a characterization of the minimizer of the functional  $\phi$ in (\ref{functional}) over $Sym^{-}$.
\begin{proposition}\label{prop}
 For $\mathbf{D}\in Sym$, there exists a unique  $\bY^* \in Sym^{-}$ satisfying the following three equivalent statements

(i) $\bY^*$  minimizes functional $\phi$ in (\ref{functional})
\begin{equation*}
\phi(\bY^*)\le \phi(\mathbf{Y}) ,
\textup{ \ for each \ }\mathbf{Y}\in Sym^{-}.
\label{funmin}
\end{equation*}

(ii) $\bY^*$  satisfies the following complementarity problem

\begin{equation}
\CC(\mathbf{D}-\bY^*)\in Sym^{+},
\label{condition1}
\end{equation}%
\begin{equation}
\bY^*\bullet \CC(\mathbf{D}-\bY^*)=0.  \label{condition2}
\end{equation}%

(iii)  $\bY^*$  satisfies the variational inequality (\ref{varineq2}).

\end{proposition}

\begin{proof}
Equivalence of (i) and (iii) follows from the minimum norm theorem \cite{FA}.
It is an easy matter to show that (ii) implies (iii).
The proof that (iii) implies (ii) is based on the fact that $Sym^{-}$ is a cone, in fact from (\ref{varineq2}), for $\mathbf{Y}=\mathbf{0}$ and for $\mathbf{Y}=2\bY^*$, we get (\ref{condition2}); condition (\ref{condition1}) follows
from (\ref{varineq2}) putting $\mathbf{Y}=\bY^*+\mathbf{Y}^\#$, with $\mathbf{Y}^\#\in Sym^{-}$. \qed
\end{proof}

The minimum point $\bY^*$ of the functional (\ref{functional}) is the projection of $
\mathbf{D}$ onto $Sym^{-}$ with respect to the inner product $\circ$ in $Sym$. Letting $P_{\CC,Sym^{-}}: Sym \rightarrow Sym^{-}$ be the nonlinear function which associates to each symmetric tensor its projection onto $Sym^{-}$ with respect to the inner product (\ref{ipC}), we have, therefore that 
\begin{equation*}
\bY^*=P_{\CC,Sym^{-}}(\mathbf{D}).
\label{Projection}
\end{equation*}%
The projection $P_{\CC, Sym^{-}}$  is monotone,
Lipschitz continuous,
and homogeneous of degree 1, i.e.
\begin{equation}
P_{\CC, Sym^{-}}(\alpha \mathbf{D})=\alpha P_{\CC,Sym^{-}}(\mathbf{D}), \textup{ \ \ }\forall\mathbf{D}\in Sym \textup{ \ \ and \ \ }\forall\alpha \ge 0,
\label{homog}
\end{equation}
and satisfies
\begin{equation}
P_{\alpha \CC, Sym^{-}}(\mathbf{D})=P_{\CC,Sym^{-}}(\mathbf{D}), \textup{ \ \ }\forall\mathbf{D}\in Sym \textup{ \ \ and \ \ }\forall\alpha > 0.
\label{boh}
\end{equation}

Moreover, it is indefinitely Fr\'echet differentiable on an open dense subset of $Sym$ \cite{PS}.

From (\ref{condition1}) and (\ref{condition2}), it follows that if $\CC(\mathbf{D}) \in Sym^+$, then $\bY^*=\mathbf{0}$, and if $\mathbf{D} \in Sym^-$, then $\bY^*=\mathbf{D}$. Moreover, it is easy to prove that when tensors $\bY^*$ and $\CC(\mathbf{D}-\bY^*)$ satisfy 
(\ref{condition1}) and (\ref{condition2}), then they commute \cite{GDP,SPRING}, i.e. 
\begin{equation*}
\bY^*\CC(\mathbf{D}-\bY^*)=\CC(\mathbf{D}-\bY^*)\bY^*=\mathbf{0}.
\label{commutation}
\end{equation*}%
 Thus $\bY^*$ and $\CC(\mathbf{D}-\bY^*)$ are coaxial \cite{GDP,SPRING} that is there exists an orthonormal basis of $\mathcal{V}$ 
 constituted by eigenvectors of both  $\bY^*$ and $\CC(\mathbf{D}-\bY^*)$.
From Proposition \ref{prop}, it follows that each tensor  $\mathbf{D} \in Sym$ can be expressed
as the following  sum
\begin{equation*}
\mathbf{D}=\bY^*+\mathbf{D}-\bY^*,
\label{somma}
\end{equation*}
where $\bY^*$ belongs to $Sym^{-}$ and $\mathbf{D}-\bY^*$ belongs to $\CC^{-1}Sym^{+}$, with
\begin{equation*}
\CC^{-1}Sym^{+}=\left\{ \mathbf{A}\textup{ : } \CC( \mathbf{A}) \in Sym^{+} \right\}.
\label{set}
\end{equation*}

%%%%%%%%%%%%%%%%%%%%%%%%%%%%%%%%%%%%%%%%%%%%%%%%%%%%%%%%%%%%%%%%%%%%%%
\section{Fourth-order tensors $\CC$ in solid mechanics} \label{Special}

We now describe some possible choices of the symmetric and positive definite tensor $\CC$ 
giving details of tensors arising when modelling masonry-like materials  that motivated this work.

When the tensor $\CC$ coincides with the identity tensor, i.e.
\begin{equation*}
\CC=\mathbb{I}_{Sym},
\label{Id}
\end{equation*}
then the  $\CC$-norm coincides with the Frobenius norm.
On the other hand, given $\mathbf{C}\in Sym$ positive definite, the fourth-order tensor defined by

\begin{equation*}
\CC(\mathbf{A})=\mathbf{C} \mathbf{A} \mathbf{C}, \,\,\, \mathbf{A} \in Sym,
\label{quad}
\end{equation*}
is symmetric and positive definite and define the weighted Frobenius norm 
\begin{equation*}
\parallel \mathbf{A}\parallel_{\CC}= \parallel \sqrt{\mathbf{C}} \mathbf{A}\sqrt{\mathbf{C}} \parallel.
\label{WFnorm}
\end{equation*}
This norm was introduced in \cite{Higham2}, where the problem of finding the nearest correlation
matrix is addressed, \cite{Higham1,Higham3}.

Other expressions for $\CC$ can be chosen within the framework of solids mechanics. In particular, minimizing functional (\ref{functional}) has  interesting applications in modelling the mechanical behaviour of masonry constructions. If one adopts the constitutive equation of
masonry-like materials \cite{GDP,SPRING} to model masonry materials, it is possible to prove that the stress $\bY^*$ associated with the infinitesimal strain $\mathbf{D}$ is the projection of  $\CC(\mathbf{D})$
onto $Sym^{-}$ with respect to the inner product defined in (\ref{ipC}), with $\CC^{-1}$ in place of $\CC$. Here  $\CC$ represents the elasticity tensor of the material and can have
several expressions depending on its different degrees of anisotropy.
In order to recall some of these expressions \cite{Gurtin2,GFA,PGV} the following definition has to be introduced.  Let $\Gamma$ be a subset of $Orth$, we say that $\CC$ is invariant under $\Gamma$ if
\begin{equation}
\CC ( \mathbf{Q}\mathbf{A}\mathbf{Q}^T)=\mathbf{Q}\CC (\mathbf{A})\mathbf{Q}^T,
\textup{ \ \ }\forall  \mathbf{A}\in Sym , \mathbf{Q}\in \Gamma.
\label{Invariant}
\end{equation}
It is an easy matter to show that if $\CC$ is invariant under $\Gamma$, the same holds for  $\CC^{-1}$.

\subsection{The isotropic case}
If $\CC$ satisfies the condition (\ref{Invariant}) with $\Gamma=Orth$, then  there exist two real numbers
and $E $ and $\nu $ such that $\CC$ has the representation
\begin{equation}
\CC=\frac{E}{1+\nu} (\mathbb{I}_{Sym}+\frac{\nu}{1-2\nu} \mathbf{I}\otimes
\mathbf{I}), \label{Isotropy}
\end{equation}
where $\mathbf{I}\in Sym$ is the identity tensor \cite{Gurtin2}.
In this case, tensor $\CC$ is called isotropic and is the elasticity tensor of an isotropic elastic material with Young's modulus $E$ and the Poisson's ratio $\nu$ \cite{GFA}. Because of (\ref{posc}), $E$ and $\nu$ satisfy the conditions%
\begin{equation}
E>0,  \textup{ \ \ }-1<\nu<1/2.
 \label{condC}
\end{equation}

We point out that if $E=1$ and $\nu=0$, tensor in (\ref{Isotropy}) is the identity tensor.
When $\CC$ has the expression in (\ref{Isotropy}), its inverse is
\begin{equation}
\CC^{-1}=\frac{1+\nu}{E} \mathbb{I}_{Sym}-\frac{\nu}{E} \mathbf{I}\otimes
\mathbf{I}. \label{Iso-1}
\end{equation}
Let us now limit ourselves to consider the isotropic fourth-order tensor $\CC$ in (\ref{Isotropy}). In this case, from the coaxiality of  tensors $\bY^*$ and $\CC(\mathbf{D}-\bY^*)$, it follows that $\mathbf{D}$ and $\bY^*$ are coaxial as well. This property makes it easy to calculate for each choice of $\mathbf{D}$ the minimum point of $\phi$ explicitly, and then compare the explicit solution to the numerical one. For the sake of comparison, the explicit solution is summarized in the following.

For $\mathbf{D}\in Sym,$ let $d_{1}\leq d_{2}\leq d_{3}$ be its
ordered eigenvalues and $\mathbf{q}_{1},$
$\mathbf{q}_{2},\mathbf{q}_{3}$ the corresponding eigenvectors. We
introduce the following tensors of $Sym$
\begin{equation}
\mathbf{O}_{11}=\mathbf{q}_{1}\otimes \mathbf{q}_{1},\textup{ }\mathbf{O}_{22}=%
\mathbf{q}_{2}\otimes \mathbf{q}_{2},\textup{ }\mathbf{O}_{33}=\mathbf{q}%
_{3}\otimes \mathbf{q}_{3}.\label{O123}
\end{equation}%

Given $\mathbf{D}$, the corresponding minimum point  $\bY^*$ of the functional $\phi$ in (\ref{functional}) is
\begin{equation}
\bY^*=\mathbf{0},\,\,\,\textup{if   }\,\,\, d_{1}+\frac{\nu}{1-2\nu}(d_1+d_2+d_3)\geq 0,
\label{T0}
\end{equation}%
\begin{equation*}
\bY^*=[d_{1}+\frac{\nu}{1-\nu}(d_2+d_3)] \mathbf{O}_{11},\,\,\, \textup{if  }\,\,\, (1-\nu)d_{1}+\nu (d_2+d_3)\leq 0,
\end{equation*}
\begin{equation}
d_2+\nu d_3\geq 0, \label{T1}
\end{equation}%
\begin{equation}
\bY^*=
(d_1+\nu d_3)\mathbf{O}_{11} +(d_2+\nu d_3)\mathbf{O}_{22},\,\,\,\textup{if  }\,\,\, d_2+\nu d_3\leq 0,
 \textup{   }d_3\geq 0,
\label{T2}
\end{equation}%
\begin{equation}
\bY^*=\mathbf{D},\,\,\,\textup{if }\,\,\, d_{3} \leq 0. \label{T3}
\end{equation}%

When $\nu=0$, tensor $\CC$ is equal to $E \mathbb{I}_{Sym}$ and the projection of $\mathbf{D}$ onto $Sym^{-}$ with respect to the inner product associated with $\CC$, is

\begin{equation*}
\bY^*= \frac{\mathbf{D}-\sqrt{\mathbf{D}^2}}{2},
\label{Tiden}
\end{equation*}
where the square root $\sqrt{\mathbf{A}}$ of the positive semidefinite symmetric tensor $\mathbf{A}$ is the unique positive semidefinite symmetric tensor $\mathbf{B}$, such that $\mathbf{B}^2=\mathbf{A}$.

\subsection{The transversely isotropic case} \label{TIcase}
If $\Gamma$ is a proper subset of $Orth$, then $\CC$ satisfying (\ref{Invariant}) is said anisotropic. In this paper, we limit our attention to
only one kind of anisotropic tensors, corresponding to the transverse isotropic materials, described in the  following.

A fourth-order tensor $\CC$ is said transversely isotropic if there exists a unit vector $\mathbf{f}_3$ (the preferred  direction of transverse isotropy) such that $\CC$ is invariant under the subgroup $\Gamma_{TI} \subset Orth$  constituted by all the rotations about $\mathbf{f}_3$,%%
\begin{equation*}
\CC(\mathbf{Q} \mathbf{A} \mathbf{Q}^T)=\mathbf{Q} \CC(\mathbf{A}) \mathbf{Q}^T, \textup{ \ \ }  \mathbf{A} \in Sym,  \mathbf{Q} \in \Gamma_{TI}.
\label{Ccubinv}
\end{equation*}
Let  $\mathsf{F}= (\mathbf{f}_{1},\mathbf{f}_{2},\mathbf{f}_{3})$ be an  orthonormal basis of  $\mathcal{V}$. If tensor $\CC$ is transversely isotropic with respect to the direction $\mathbf{f}_{3}$, then $\CC$ has the following representation \cite{PGV}
\begin{equation}
\CC=\sum_{i=1}^5  {\alpha_i \CC_i},
\label{CTI}
\end{equation}
where
\begin{equation}
\CC_1=\mathbf{R} \otimes \mathbf{R}, \textup{ \ \ } \CC_2=\mathbf{Q} \otimes \mathbf{Q},\textup{ \ \ } \CC_3=\mathbf{R} \otimes \mathbf{Q}+\mathbf{Q} \otimes \mathbf{R},
\label{CTI1}
\end{equation}
\begin{equation}
\CC_4=4\mathbf{R} \odot \mathbf{Q}, \textup{ \ \ } \CC_5=2\mathbf{Q} \odot \mathbf{Q}-\CC_2,
\label{CTI2}
\end{equation}
with $\mathbf{R}=\mathbf{f}_3 \otimes \mathbf{f}_3$ and $\mathbf{Q}=\mathbf{I}-\mathbf{R}$. The real numbers $\alpha_i$  satisfy the conditions
\begin{equation*}
\alpha_4>0,  \textup{ \ \ }\alpha_5>0, \textup{ \ \ } \alpha_1+2 \alpha_2 >0,   \textup{ \ \ }  \alpha_1 \alpha_2-\alpha_3^2>0,
\label{condCTI}
\end{equation*}
which guarantee the positive definiteness of  $\CC$ in (\ref{CTI}). Tensor $\CC$ in (\ref{CTI}) describes the mechanical behaviour of a transversely isotropic elastic material \cite{Gurtin2,PGV}.
For $\alpha_3=2 \alpha_2-\alpha_1$ and $\alpha_4=\alpha_5=\alpha_1-\alpha_2$ the fourth-order tensor in (\ref{CTI}) becomes isotropic,
\begin{equation*}
\CC=2 (\alpha_1-\alpha_2)\mathbb{I}_{Sym}+(2 \alpha_2-\alpha_1) \mathbf{I}\otimes
\mathbf{I}, \label{Isobis}
\end{equation*}
with
\begin{equation*}
\alpha_1=\frac{(1-\nu)E}{(1+\nu)(1-2\nu)},  \textup{ \ \ } \alpha_2=\frac{E}{2(1+\nu)(1-2\nu)}.
\label{alpha12}
\end{equation*}

In the anisotropic case $\mathbf{D}$ and $\bY^*$ are not coaxial and the the minimum point of functional $\phi$ can be calculated explicitly only for a few choices of $\mathbf{D}$. In particular, when $\mathbf{D}$ has the eigenvector $\mathbf{f}_3$, then $\mathbf{D}$, $\bY^*$ and $\CC(\mathbf{D}-\bY^*)$ are coaxial and the solution if provided in the following.

Given $\mathbf{D}\in Sym,$ let $d_{1}\leq d_{2}\leq d_{3}$ be its
ordered eigenvalues and $\mathbf{q}_{1},$
$\mathbf{q}_{2},\mathbf{f}_{3}$ the corresponding eigenvectors. For $\mathbf{O}_{11}$ and $\mathbf{O}_{22}$ defined in (\ref{O123}), the solution $\bY^*$ has the following expressions.
\begin{equation*}
\bY^*=\mathbf{0},\,\,\,\textup{if   }\,\,\, (\alpha_2+\alpha_5)d_1+(\alpha_2-\alpha_5)d_2+\alpha_3 d_3\geq 0,
\end{equation*}
\begin{equation}
\alpha_3(d_1+d_2)+\alpha_1 d_3\geq 0, \label{TI0}
\end{equation}

\begin{equation*}
\bY^*=[d_{1}+\frac{(\alpha_2-\alpha_5)d_2+\alpha_3d_3}{\alpha_2+\alpha_5}] \mathbf{O}_{11},\,\,\, \textup{if  }\,\,\, (\alpha_2+\alpha_5)d_1+(\alpha_2-\alpha_5)d_2+\alpha_3 d_3\leq 0,
\end{equation*}
\begin{equation}
2\alpha_2d_2+\alpha_3d_e\geq 0, \,\,\, 2\alpha_3\alpha_5d_2+[\alpha_1(\alpha_2+\alpha_5)-\alpha_3^2]d_3 \geq 0, \label{TI1}
\end{equation}%

\begin{equation}
\bY^*=
[d_1+\frac{\alpha_3}{2\alpha_2}d_3]\mathbf{O}_{11} +[d_2+\frac{\alpha_3}{2\alpha_2}d_3]\mathbf{O}_{22},\,\,\,\textup{if  }\,\,\, 2\alpha_2d_2+\alpha_3 d_3\leq 0,\,\,\, d_3 \geq 0,
\label{TI2}
\end{equation}%

\begin{equation}
\bY^*=\mathbf{D},\,\,\,\textup{if }\,\,\, d_{3} \leq 0. \label{TI3}
\end{equation}%
In general, suitable numerical strategies  should be adopted to calculate the minimum point of $\phi$, as proposed in the next section.

\section{A primal-dual path following interior-point method}\label{numeric}
The numerical computation of the minimum point of functional (\ref{functional}) can be
efficiently performed by exploiting the characterization of the solutions described in (\ref{condition1})
and (\ref{condition2}). Indeed, setting $\bX=-\bY$ these conditions describe 
the following monotone semidefinite linear complementarity problem (SDLCP) in the space of symmetric tensors
\begin{equation}\label{SDLPC}
(\bX,\bS) \in \AA,\ \bX \in Sym^+,\ \bS \in Sym^+ \mbox{ and } \bX \bullet \bS = 0,
\end{equation}
where the affine subspace $\AA \subseteq Sym \times Sym $ is given by
$$
\AA = \{ (\bX,\bS) \in Sym \times Sym\ | \  \bS = \cal{C}(\bD+\bX) \}.$$
The subspace $\AA$ has dimension 6 and is monotone as 
$$ (\bX - \bX') \bullet (\bS - \bS') = \|\bX - \bX' \|^2_{\cal{C}} \ge 0,$$
for all $ (\bX',\bS') $ and $ (\bX,\bS) \in \AA$.
\begin{comment}
We call $(\bX,\bS) \in \AA$ with 
 $\bX \in Sym^+$ and $\bS \in Sym^+$ a {\em feasible solution}
 of problem  \req{SDLPC} and  $(\bX,\bS) \in \mathcal {A}$ with 
 $\bX \in Sym^{++}$ and $\bS \in Sym^{++}$ a {\em strictly feasible solution}.
 \Commentcp{forse queste definizioni si possono togliere, non vengono mai richiamate}
\end{comment}

We now describe an interior point algorithm for the efficient and accurate solution of \req{SDLPC}  that exploits the form of the tensors $\CC$ presented in the previous section.
In  \cite{SDLCPpc,SDLCPdir,SDLCP1}, the general theoretical framework of the algorithm is given 
for SDLCPs with affine subspaces $\AA$ of a general form but our practical implementation is based on that of interior-point approaches in \cite{QSDP2,QSDP1} for a somehow related problem, that is the solution of convex Quadratic SemiDefinte Programming (QSDP) problems.

Problem \req{SDLPC}, and equivalently \req{condition1} and \req{condition2} with $\bY=-\bX$, 
are in fact the first-order optimality conditions of the following QSDP problem
\begin{equation}\label{sdp_primal}
  \begin{array}{ll}
     \min_{\bX} & p(\bX) = \frac{1}{2} \bX \bullet \CC(\bX) + \bX \bullet \CC(\bD) \\
     \mbox{s.t. } &  \bX \in Sym^+. \\
    \end{array}
\end{equation}
%
%where $\bX \in Sym$ is unknown. 
This problem differs with respect to the standard formulation
of QSDPs as only positive semidefiniteness constraints are present
while the usual linear constraints are not included, see e.g. \cite{QSDP1}.
Clearly, the functionals $p(\bX)$ and  $\phi(\bY)$ defined in \req{functional} and \req{sdp_primal} respectively, 
have the same minimizers (up to the sign).

Several methods have been proposed in the literature for  standard  QSDPs
ranging from interior point methods \cite{QSDP2,QSDP1} to semismooth Newton approaches \cite{QSDPNAL}, 
passing through reformulations as a standard semidefinite-quadratic-linear program  \cite{Anjos}.
%) by introducing an additional  linear constraints, 4 artificial variables and a second order cone constraint \cite{}.
Most of these works focus on the design and analysis of  efficient algorithms  for
the case where the matrix dimensions and/or the number of linear constraints are large,
and it may be impossible to explicitly store or compute the matrix representation of $\CC$.
Conversely, here we are interested in the accurate solution of problem 
(\ref{sdp_primal})  in the small case setting and propose to use a primal-dual path-following interior point method in the spirit of \cite{QSDP2,QSDP1}.

\begin{comment}
a general theoretical framework of interior-point methods for the monotone SDLCP is given.
In particular,\cite[Theorem 3.1.]{SLCP} ensures that  the central path leading to a solution of the problem exists
assuming that the interior of the feasible region is nonempty.

We now describe an interior-point methods for QSDPs adapted to problem (\ref{sdp_primal}) by exploiting the special structure of the tensor $\CC$ discussed in the previous section.
\end{comment}

We now describe the main steps of an interior point method based on the
primal-dual path-following method given in \cite{QSDP2,QSDP1}
for the solution of (\ref{sdp_primal}) coupled with its dual form:
\begin{equation}\label{sdp_dual}
  \begin{array}{ll}
  \max_{\bX,\bS} & d(\bX) = -\frac{1}{2} \bX \bullet \CC(\bX) \\
     \mbox{s.t. } &  \bS =  \CC(\bD+\bX) \\
     & \bS \in Sym^+.
  \end{array}
\end{equation}
%
%The unknowns $\bX$ and $\bS$ are called the ``primal'' and ``dual'' variables,
%respectively. 
The algorithm, named \namedot, uses Mehrotra's predictor-corrector steps
and practical details are postponed to Section \ref{algo}.
\name is based on approximating a sequence of points
on the central path. The central path  is defined as the set of solutions
$(\bX_{\mu}, \bS_{\mu})$
to the central path equations
\begin{equation}\label{F_mu}
F_{\mu}(\bX,\bS)=  \left (\begin{array}{c}
 \bS -\CC(\bX +\bD)\\
 \bX \bS- \sigma \mu \bI
\end{array}
\right )=\mathbf{0}, \qquad \bX \in Sym^{++},\  \bS \in Sym^{++},
\end{equation}
where $\sigma \in [0,1]$ is the   centering parameter and $\mu$ is the duality measure defined by 
$$
\mu = \frac{\bX  \bullet \bS}{3}.
$$
%which measures the average value of the pairwise products xisi.

Equations (\ref{F_mu}) can be also interpreted as the perturbed first-order optimality conditions for problems (\ref{sdp_primal})-(\ref{sdp_dual}).
Fixed $\sigma$ and assuming that there exists  $(\bX,\bS) \in \AA$ with 
$\bX \in Sym^{++}$ and $\bS \in Sym^{++}$, then 
\cite[Theorem 3.1]{SDLCP1} ensures that for every $\mu  > 0$, 
there exists a unique $(\bX_\mu,\bS_\mu)$ that lies on the central path, that is that solves \req{F_mu}.

Note that the first block equation in $F_{\mu}$ above is linear, while the second is
mildly nonlinear. Hence a Newton step seems a natural idea for an
iterative algorithm. Unfortunately, the residual map $F_\mu$
takes an iterate $(\bX, \bS) \in  Sym \times Sym $
to a point in $ Sym\times Lin $ (since $\bX \bS - \mu \mathbf{I}$ is in general not symmetric), which
is a space of higher dimension, and so Newton's method cannot be applied
directly.
To apply Newton-type algorithms it is previously necessary to symmetrize
the term $\bX  \bS$ so that the resulting equivalent nonlinear system gives a function that maps
$ Sym \times Sym $.  A popular and effective technique to overcome this issue, is introducing general symmetrization scheme based on the fourth-order tensor $\HH_{\mathbf{P}} :  Lin \rightarrow  Sym$
defined as
$$
\HH_{\mathbf{P}}= {\mathbf{P}} \odot {\mathbf{P}}^{-T}
$$
where ${\mathbf{P}}$ is some nonsingular tensor, see \cite{ToddDir} and references therein.

It has been shown that for any nonsingular tensor ${\mathbf{P}}$, the system 
$ F_{\mu}(\bX ,\bS)= 0 $ in (\ref{F_mu}) is equivalent to the system
\begin{equation}\label{F_mu_sym}
\tilde F_{\mu}(\bX,\bS)=  \left (\begin{array}{c}
 \bS - \CC(\bX+\bD)\\
\HH_{\mathbf{P}}(\bX \bS)- \sigma \mu \bI
\end{array}
\right )=\mathbf{0},
\end{equation}
to which a Newton-type method can be applied.
% Several interior-point methods proposed in the literature
% are Newton steps for nonlinear system of the form
% \begin{equation}\label{F_mu_sym}
% \tilde F_{\sigma\mu}(X,y,S)=  \left (\begin{array}{c}
% -\CC(\bX ) + S - \CC(\bD)\\
% \Theta(\bX ,S)- \sigma \mu \bI
% \end{array}
% \right )=0,
% \end{equation}
% where the last equation is some symmetrization of  $\bX S$.
% Let  $\bRd$ and $\bRc$ be the dual residual and the complementarity gap at the current iterate $(\bX _k,  S_k)$
% \begin{eqnarray}
%  \bRd & = & -\CC[\bX _k]  + S_k -  \CC(\bD)\\
%  \bRc & = & \Theta(\bX _k, S_k) - \sigma \mu \bI
% \end{eqnarray}
% and
Fixed $\mu$ and given the current iteration $(\bX ,   \bS)$, let  $(\Delta \bX , \Delta \bS)\in Sym \times  Sym $ denote a Newton direction.
Then it satisfies
\begin{eqnarray}
 -\CC(\Delta \bX)+ \Delta \bS & = & \bRd    \label{New1}\\
  \EE (\Delta \bX ) + \FF (\Delta \bS) & = & \bRc  \label{New2}
\end{eqnarray}
where the fourth-order tensors $\EE=\EE(\bX ,\bS)$ and $\FF=\FF(\bX ,\bS)$ are the derivative
of $\HH_{\mathbf{P}}$ with respect to $\bX $ and $\bS$ respectively, evaluated at the current iterate, i.e.
$$
\EE = {\mathbf{P}} \odot ({\mathbf{P}}^{-T}\bS) \qquad \FF = ({\mathbf{P}}\bX ) \odot {\mathbf{P}}^{-T},
$$
and $\bRd$ and $\bRc$ are the current dual residual and complementarity gap given by
\begin{eqnarray}
 \bRd & = & \CC(\bX +\bD)  - \bS , \label{Rd}\\
 \bRc & = &  \sigma \mu \bI - \HH_{\mathbf{P}}(\bX  \bS). \label{Rc}
\end{eqnarray}
Depending on the choice of $\mathbf{P}$, the tensors  $\EE$ and $\FF$ have a different form and therefore different forms 
for the equations (\ref{New1})-(\ref{New2}) can be  derived.
%In any case, the solution of (at least) one linear system with an $\bar n \times \bar n$ matrix is required.

Several choices for $\mathbf{P}$ are available in the literature \cite{ToddDir}. One of the most popular choice  yields the so-called  Nesterov-Todd  (NT) direction \cite{NT} and is obtained by 
choosing ${\mathbf{P}} = \bW^{-1/2}$ with $\bW$ being the geometric mean of $\bX $ and $\bS^{-1}$, i.e.
\begin{equation}\label{eq:W}
\bW= \bX ^{1/2}(\bX ^{1/2}\bS \bX ^{1/2})^{-1/2} \bX ^{1/2} =
                 \bS^{-1/2}(\bS^{1/2}\bX  \bS^{1/2})^{1/2} \bS^{-1/2},
                 \end{equation}
                 ($\bW\bS\bW=\bX $). The corresponding forms of  $\EE$ and $\FF$ and $\bRc$ are:
  $$\EE = \bS \odot \bW^{-1}, \quad \FF = \bW \odot \bW,\quad \bRc= \mu \bS^{-1} - \bX .
$$
Assuming that $\EE$ is nonsingular, system (\ref{New1})-(\ref{New2}) has a unique solution if and only if the  Schur complement 
$$ \SS =  \EE^{-1} \FF + \CC^{-1},$$
is positive definite \cite{NT}. This condition holds when $\bX, \bS$ and $\HH_{\mathbf{P}}(\bX \bS)$ are positive semidefinite.  In particular, $\HH_{\mathbf{P}}$ is positive definite whenever $\mathbf{P}$ is an invertible tensor that satisfies $\mathbf{P}^T \mathbf{P} = \bW^{-1}$ where $\bW$ is such that $\bW  \bS \bW = \bS$, as for the NT direction, see  \cite{NT}.

The  application of the method in \cite{QSDP2,QSDP1} to problems \req{sdp_primal}-\req{sdp_dual} yields the following procedure for the solution of the Newton system
\req{New1}-\req{New2}: 
solve the Schur complement system
\begin{equation}\label{sisX}
\SS (\Delta \bS ) =
\EE^{-1}(\bRc) - \CC^{-1} (\bRc),
\end{equation}
and compute $\Delta \bX   =  \CC^{-1}(\bRd - \Delta \bS)$.

For the NT direction, the  Schur complement tensor takes the form
\begin{equation}\label{Snt}
\SS_{NT} = (\bW \odot \bW) + \CC^{-1},\end{equation}
where $\bW$ is given in \req{eq:W} and   we recall that $\CC$ is nonsingular on $Sym$ and that the inverse $\CC^{-1}$ is explicitly available in the applications considered in this work.

%$$\SS_{AHO} = (\bI \odot \bS)^{-1} (\bX  \odot \bI) + \CC^{-1}. $$

In the next section we will show that the solution of \req{sisX} with
$\SS_{NT}$ may yield an inaccurate solution of the original problem
\req{SDLPC} due to the fast increasing of the condition number of
$\SS_{NT}$ as $\mu$ goes to zero. 

In order to provide an accurate solution of problem \req{SDLPC}, we propose
to use an alternative choice of the tensor ${\mathbf{P}}$ that was firstly proposed in \cite{AHO} and 
yields the so-called Alizadeh-Haeberly-Overton (AHO) direction.  
The AHO direction corresponds to set ${\mathbf{P}} = \bI$
that gives 
$$\EE = \bI \odot \bS, \quad \FF = \bX  \odot \bI,\quad \bRc= \mu \bI -\frac{1}{2}(\bX \bS+\bS\bX ).
 $$
We observe that with this choice of ${\mathbf{P}}$, the Newton system \req{New1}-\req{New2} admits a unique solution when 
 and $\bX,\bS$ are positive semidefinite  and $\bX \bS + \bS \bX $ is positive definite \cite[Corollary 3.2]{SDLCP1}.
%%%%%%%%%%%%%%%%%%

As alternative to the linear system with the Schur complement in \req{sisX}, 
in this work we propose to solve a different linear system that for the AHO direction possesses 
optimal conditioning properties when the tensor $\CC$ is of the form \req{Isotropy} or \req{CTI}-\req{CTI2} discussed in Section \ref{Special}.
Indeed, a solution of
\req{New1}-\req{New2} can be obtained also by
computing $\Delta \bX$ from
\begin{equation}\label{sisX1}
 \MM(\Delta \bX) =  \bRc - \FF (\bRd)
\end{equation}
where
$$
 \MM = \EE + \FF \CC,
$$
and then retrieving $\Delta \bS $ form \req{New1}. In fact, for the AHO direction, the above tensor 
$\MM$ has the following special form
\begin{equation}\label{Maho}
\MM_{AHO} =  (\bI \odot \bS) + (\bX  \odot \bI) \CC.
\end{equation}
We observe that while the Schur complement $\SS$ is symmetric the tensor $\MM$ is in general nonsymmetric.
The conditioning properties of $\SS_{NT}$ and $\MM_{AHO}$ are discussed in the next section.

\subsection{Conditioning issues} \label{condSM}

Let ${\mu^{(k)}}$ be a monotonically decreasing sequence such that $\lim_{k\rightarrow \infty} \mu^{(k)} =0$ 
and let $(\bX^{(k)}, \bS^{(k)})$ be a point on the central path corresponding to $\mu^{(k)}$, that is
$(\bX^{(k)}, \bS^{(k)})$ satisfies \req{F_mu}. Moreover, let $\SS_{NT}^{(k)}$ and $\MM_{AHO}^{(k)}$ be the corresponding tensors of the Newton systems
given in \req{Snt} and \req{Maho}, respectively. 
%Let $\lambda_{\min}(\CC^{-1})$ and $\lambda_{\max}(\CC^{-1})$ denote the minimum and maximum eigenvalue of $\CC^{-1}$, respectively.

Assume that the sequence $(\bX^{(k)}, \bS^{(k)})$ converges to the optimal solution $(\bX^*, \bS^*)$ as $\mu^{(k)}$  tends to zero  and that the ranks of $\bX^*$ and $\bS^*$ sum up to $3$.

We will now show that under these conditions, the condition number of $\SS^{(k)}_{NT}$ may not be bounded for $\mu^{(k)} \rightarrow 0$ while the condition number of $\MM^{(k)}_{AHO}$ does not depend on $\mu^{(k)}$ when $\CC$ is the isotropic tensor in \req{Isotropy}. Moreover we conjecture that the condition number of
$\MM^{(k)}_{AHO}$ is still bounded in the transversely isotropic case in \req{CTI}.

Let $\lambda_1^{(k)}, \lambda_2^{(k)}, \lambda_3^{(k)}$  and $\xi_1^{(k)}, \xi_2^{(k)}, \xi_3^{(k)}$
be the eigenvalues of $\bX^{(k)}$ and $\bS^{(k)}$, respectively. We observe that $\bX^{(k)}$ and $\bS^{(k)}$ commute and we denote by $(\mathbf{q}_1^{(k)}, \mathbf{q}_2^{(k)},\\ \mathbf{q}_3^{(k)})$ a  basis of common eigenvectors, moreover it holds  
\begin{equation}\label{gapk}
    \lambda_i^{(k)}\xi_i^{(k)} =  \sigma\mu^{(k)}, \mbox{  for } i = 1,2,3.
\end{equation}
Let $\lambda_1^{*}, \lambda_2^{*}, \lambda_3^{*}$  and $\xi_1^{*}, \xi_2^{*}, \xi_3^{*}$
be the eigenvalues of $\bX^*$ and $\bS^*$, respectively. They satisfy $\lambda_i^*\xi_i^* =  0$ and since $rank(\bX^*)+rank(\bS^*) = 3$, only the following 4 cases can occur:
\begin{itemize}
    \item case 1: $\xi_1^*, \xi_2^*,\xi_3^* >0$, $\lambda_1^*=\lambda_2^*=\lambda_3^*=0$;
    \item case 2: $\xi_1^*, \xi_2^* >0$ and $\xi^*_3=0$, $\lambda_1^*=\lambda_2^*=0$ and $\lambda_3^*>0$;
    \item case 3: $\xi_1^* >0$ and $\xi_2^*=\xi^*_3=0$, $\lambda_1^*=0$ and $\lambda_2^*, \lambda_3^*>0$;
    \item case 4: $\xi_1^*=\xi_2^*=\xi_3^*=0$, $\lambda_1^*,\lambda_2^*,\lambda_3^*>0$.
\end{itemize}

\begin{comment}
For $k$ sufficiently large, $\bW^{(k)}= \bQ^{(k)} {\mathbf{D}}^{(k)}\bQ^{(k)}^T$ with ${\mathbf{D}}^{(k)} = {\mathbf{\Lambda}}^{(k)}^{1/2}{\mathbf{\Xi}}^{(k)}^{-1/2}$
has $r$ and $s$ eigenvalues of the orders $O(1/\sqrt{\mu^{(k)}})$ and
$O (\sqrt{\mu^{(k)}})$, respectively.

\end{comment}

We now compute the eigenvalues of $\bW^{(k)} \odot \bW^{(k)}$.
Since $\bS^{(k)}$ and $\bX^{(k)}$ lie on the central path, by the definition of
$\bW^{(k)}$ in \req{eq:W} we have that the eigenvalues of
$\bW^{(k)}$ are $\sqrt{\lambda_i^{(k)}}/\sqrt{\xi_i^{(k)}}$ for $i = 1,2,3$.

Therefore, the eigenvalues of $\bW^{(k)} \odot \bW^{(k)}$  are $\frac{\sqrt{\lambda_i^{(k)} \lambda_j^{(k)}}}{\sqrt{\xi_i^{(k)} \xi_j^{(k)}}}$ for $i,j=1,2,3$, $i \le j$; thus, taking into account \req{gapk} the following cases can occur:
\begin{itemize}
  \item  case 1: $\bS^{(k)}$ has eigenvalues $\xi_1^{(k)}$, $\xi_2^{(k)}$, $\xi_3^{(k)}$ and $\bX^{(k)}$ has eigenvalues $\frac{\sigma\mu^{(k)}}{\xi_1^{(k)}}$, $\frac{\sigma\mu^{(k)}}{\xi_2^{(k)}}$, $\frac{\sigma\mu^{(k)}}{\xi_3^{(k)}}$, then the eigenvalues of $\bW^{(k)} \odot \bW^{(k)}$ are
\begin{equation*}
      \frac{\sigma\mu^{(k)}}{\xi_i^{(k)} \xi_j^{(k)}},\,\,\, i, j=1,2,3, \,\,\, i \le j.
\end{equation*}
  \item case 2:  $\bS^{(k)}$ has eigenvalues $\xi_1^{(k)}$, $\xi_2^{(k)}$, $\frac{\sigma\mu^{(k)}}{\lambda_3^{(k)}}$ and $\bX^{(k)}$ has eigenvalues $\frac{\sigma\mu^{(k)}}{\xi_1^{(k)}}$, $\frac{\sigma\mu^{(k)}}{\xi_2^{(k)}}$, $\lambda_3^{(k)}$, then the eigenvalues of $\bW^{(k)} \odot \bW^{(k)}$ are
\begin{equation*}
     \frac{\sigma\mu^{(k)}}{(\xi_1^{(k)})^2},\,\,\, \frac{\sigma\mu^{(k)}}{(\xi_2^{(k)})^2}, \,\,\, \frac{\sigma\mu^{(k)}}{\xi_1^{(k)}\xi_2^{(k)}}, \frac{(\lambda_3^{(k)})^2}{\sigma\mu^{(k)}},\,\,\,\frac{\lambda_3^{(k)}}{\xi_1^{(k)}}, \,\,\,\frac{\lambda_3^{(k)}}{\xi_2^{(k)}}.
\end{equation*}
  \item case 3:  $\bS^{(k)}$ has eigenvalues $\xi_1^{(k)}$, $\frac{\sigma\mu^{(k)}}{\lambda_2^{(k)}}$, $\frac{\sigma\mu^{(k)}}{\lambda_3^{(k)}}$ and $\bX^{(k)}$ has eigenvalues $\frac{\sigma\mu^{(k)}}{\xi_1^{(k)}}$, $\lambda_2^{(k)}$, $\lambda_3^{(k)}$, then the eigenvalues of $\bW^{(k)} \odot \bW^{(k)}$ are
\begin{equation*}
    \frac{\sigma\mu^{(k)}}{(\xi_1^{(k)})^2},\,\,\, \frac{(\lambda_2^{(k)})^2}{\sigma\mu^{(k)}},\,\,\,\frac{(\lambda_3^{(k)})^2}{\sigma\mu^{(k)}}, \,\,\, \frac{\lambda_2^{(k)}}{\xi_1^{(k)}},\,\,\,\frac{\lambda_3^{(k)}}{\xi_1^{(k)}},\,\,\, \frac{\lambda_1^{(k)} \lambda_2^{(k)}}{\sigma\mu^{(k)}}
\end{equation*}
  \item case 4:  $\bS^{(k)}$ has eigenvalues $\frac{\sigma\mu^{(k)}}{\lambda_1^{(k)}}$, $\frac{\sigma\mu^{(k)}}{\lambda_2^{(k)}}$, $\frac{\sigma\mu^{(k)}}{\lambda_3^{(k)}}$ and $\bX^{(k)}$ has eigenvalues $\lambda_1^{(k)}$, $\lambda_2^{(k)}$, $\lambda_3^{(k)}$, then the eigenvalues of $\bW^{(k)} \odot \bW^{(k)}$ are

\begin{equation*}
     \frac{(\lambda_1^{(k)})^2}{\sigma\mu^{(k)}},\,\,\,\frac{(\lambda_2^{(k)})^2}{\sigma\mu^{(k)}},\,\,\,\frac{(\lambda_3^{(k)})^2}{\sigma\mu^{(k)}},\,\,\, \frac{\lambda_1^{(k)} \lambda_2^{(k)}}{\sigma\mu^{(k)}},
     \,\,\,\ \frac{\lambda_1^{(k)} \lambda_3^{(k)}}{\sigma\mu^{(k)}}, \,\,\ \frac{\lambda_2^{(k)} \lambda_3^{(k)}}{\sigma\mu^{(k)}}
\end{equation*}
\end{itemize}

%!!!!!!!!!!!!!!!!!!!!!!!!!!!!!!!!!!!!!!!!!!!!!!!!!!!!!!!!!!!!!!!!!!!!!!!!!!!!
%This implies that $\bW_k \odot \bW_k$ has $\bar r$, $rs$, and $\bar s$ eigenvalues of the orders  $O(1/\mu_k)$, $O(1)$ and
%$O(\mu_k)$, respectively(vedi appendice \cite{NT}.%
%!!!!!!!!!!!!!!!!!!!!!!!!!!!!!!!!!!!!!!!!!!!!!!!!!!!!!!!!!!!!!!!!!!!!!!!!!!!!
 From the Courant-Fisher-Weyl min-max principle \cite[Corollary 3.13]{{Weyl}} we get:
 $$
\lambda_{\max} (\bW^{(k)}\odot \bW^{(k)}) + \lambda_{\min}(\CC^{-1}) \le  \lambda_{\max}(\SS^{(k)}_{NT} ) \le \lambda_{\max} (\bW^{(k)}\odot \bW^{(k)}) + \lambda_{\max}(\CC^{-1})
 $$
 and 
  $$
 \lambda_{\min} (\bW^{(k)}\odot \bW^{(k)}) + \lambda_{\min}(\CC^{-1}) \le \lambda_{\min}(\SS^{(k)}_{NT}) \le  \lambda_{\min} (\bW^{(k)}\odot\bW^{(k)}) + \lambda_{\max}(\CC^{-1}) 
 $$
 and the condition number $\kappa(\SS^{(k)}_{NT})$ of $\SS^{(k)}_{NT}$ satisfies the inequality:
  $$
\frac{\lambda_{\max} (\bW^{(k)}\odot \bW^{(k)}) + \lambda_{\min}(\CC^{-1})}{\lambda_{\min} (\bW^{(k)}\odot \bW^{(k)}) + \lambda_{\max}(\CC^{-1}) } \le  \kappa(\SS^{(k)}_{NT} ) 
\le \frac{\lambda_{\max} (\bW^{(k)}\odot \bW^{(k)}) + \lambda_{\max}(\CC^{-1})}{\lambda_{\min} (\bW^{(k)}\odot \bW^{(k)}) + \lambda_{\min}(\CC^{-1})}.
 $$
Therefore in cases 1 and 4 $\kappa(\SS^{(k)}_{NT})$ is bounded, while in cases 2 and 3 we have that $\kappa(\SS^{(k)}_{NT}) = O((\mu^{(k)})^{-1})$.

Let us now consider the nonsymmetric fourth-order tensor $\MM_{AHO}^{(k)}$ defined in \req{Maho} with $\bS^{(k)}$ and $\bX^{(k)}$ in the place of $\bS$ and $\bX$. In order to analyze its condition number, we consider the positive definite symmetric fourth-order tensor 
$$\LL^{(k)}=\MM_{AHO}^{(k)} (\MM_{AHO}^{(k)})^T,$$ 
which has the form
\begin{equation}\label{K}
\LL^{(k)} = (\bI \odot \bS^{(k)})^2 
+ (\bX^{(k)}  \odot \bI) \CC(\bI \odot \bS^{(k)}) + (\bI \odot \bS^{(k)}) \CC (\bX^{(k)}  \odot \bI) + (\bX^{(k)}  \odot \bI) \CC^2 (\bX^{(k)}  \odot \bI)
\end{equation}
and calculate its eigenvalues focusing on the case in which $\CC$ is isotropic with expression (\ref{Isotropy}). By introducing the Lam\a'{e} moduli
\begin{equation*}
\psi=\frac{E}{2(1+\nu)}, \,\,\, \omega=\frac{\nu E}{(1+\nu)(1-2\nu)},
\end{equation*}
from \req{K} we get for $\mathbf{H} \in Sym$
\begin{eqnarray}
\LL^{(k)}(\mathbf{H}) & =& \frac{1}{4}\left(\mathbf{H}(\bS^{(k)})^2+(\bS^{(k)})^2\mathbf{H}+2\bS^{(k)}\mathbf{H}\bS^{(k)}\right)+ \nonumber\\
& & \psi^2\left(\mathbf{H}(\bX^{(k)})^2+(\bX^{(k)})^2\mathbf{H}+2\bX^{(k)}\mathbf{H}\bX^{(k)}\right)+  \nonumber\\
& & \frac{\psi}{2}\left(2\bX^{(k)}\mathbf{H}\bS^{(k)}+2\bS^{(k)}\mathbf{H}\bX^{(k)}+\bX^{(k)}\bS^{(k)}\mathbf{H} + \right .  \nonumber\\
& & \left .\bS^{(k)}\bX^{(k)}\mathbf{H}+\mathbf{H}\bX^{(k)}\bS^{(k)}+\mathbf{H}\bS^{(k)}\bX^{(k)}\right)+  \nonumber\\
& & \omega\left(tr(\mathbf{H}\bS^{(k)})\bX^{(k)}+tr(\mathbf{H}\bX^{(k)})\bS^{(k)}+(3\omega+4\psi)tr(\mathbf{H}\bX^{(k)})\bX^{(k)}\right). \label{K_H}
\end{eqnarray}

It is easy to show that the following three positive real numbers
\begin{equation}
m_1^{(k)}=\frac{1}{4}(\xi_1^{(k)}+\xi_2^{(k)})^2+2\psi^2(\lambda_1^{(k)}+\lambda_2^{(k)})^2+(\xi_1^{(k)}+\xi_2^{(k)})(\lambda_1^{(k)}+\lambda_2^{(k)}), \label{m1}
\end{equation}
\begin{equation}
m_2^{(k)}=\frac{1}{4}(\xi_1^{(k)}+\xi_3^{(k)})^2+2\psi^2(\lambda_1^{(k)}+\lambda_3^{(k)})^2+(\xi_1^{(k)}+\xi_3^{(k)})(\lambda_1^{(k)}+\lambda_3^{(k)}),\label{m2}
\end{equation}
\begin{equation}
m_3^{(k)}=\frac{1}{4}(\xi_2^{(k)}+\xi_3^{(k)})^2+2\psi^2(\lambda_2^{(k)}+\lambda_3^{(k)})^2+(\xi_2^{(k)}+\xi_3^{(k)})(\lambda_2^{(k)}+\lambda_3^{(k)}),\label{m3}
\end{equation}
are eigenvalues of $\LL^{(k)}$ with eigentensors
\begin{equation}
\frac{1}{\sqrt{2}}(\mathbf{q}_1^{(k)}\otimes\mathbf{q}_2^{(k)}+\mathbf{q}_2^{(k)}\otimes\mathbf{q}_1^{(k)}), \,\,\, \frac{1}{\sqrt{2}}(\mathbf{q}_1^{(k)}\otimes\mathbf{q}_3^{(k)}+\mathbf{q}_3^{(k)}\otimes\mathbf{q}_1^{(k)}),\label{q1q2}
\end{equation}
\begin{equation}
\frac{1}{\sqrt{2}}(\mathbf{q}_2^{(k)}\otimes\mathbf{q}_3^{(k)}+\mathbf{q}_3^{(k)}\otimes\mathbf{q}_2^{(k)}).\label{q3}
\end{equation}

The remaining eigentensors  belong to the subspace of $Sym$ spanned by $\mathbf{q}_1^{(k)}\otimes\mathbf{q}_1^{(k)}$,
$\mathbf{q}_2^{(k)}\otimes\mathbf{q}_2^{(k)}$ and
$\mathbf{q}_3^{(k)}\otimes\mathbf{q}_3^{(k)}$.

Thus, we look for real numbers $m$ and triples $(\chi_1, \chi_2, \chi_3) \ne (0, 0, 0)$ such that, putting
\begin{equation*}
\mathbf{A}=\chi_1 \mathbf{q}_1^{(k)}\otimes\mathbf{q}_1^{(k)}+\chi_2 \mathbf{q}_2^{(k)}\otimes\mathbf{q}_2^{(k)}+\chi_3 \mathbf{q}_3^{(k)}\otimes\mathbf{q}_3^{(k)}
\end{equation*}
we have
\begin{equation}
\LL^{(k)}(\mathbf{A})=m \mathbf{A}. \label{SysAHO}
\end{equation}
From \req{K_H} we get
\begin{eqnarray*}
\LL^{(k)}(\mathbf{q}_i^{(k)}\otimes\mathbf{q}_i^{(k)})&=&(\xi_i^{(k)}+2\psi\lambda_i^{(k)})^2\mathbf{q}_i^{(k)}\otimes\mathbf{q}_i^{(k)} + \\
& & \omega \left [\left(\xi_i^{(k)}+(3\omega+4\psi)\lambda_i^{(k)}\right)\bX^{(k)}+\lambda_i\bS^{(k)}\right],
\end{eqnarray*}
and, taking into account the linear independence of tensors $\mathbf{q}_i^{(k)}\otimes\mathbf{q}_i^{(k)}$, $i=1,2,3$, we can conclude that  nonzero triples $(\chi_1, \chi_2, \chi_3)$ exist provided that $m$ is a root of  the following third-degree polynomial
\begin{equation}
\tilde{p}_k(m)=m^3+a_k m^2+b_k m+ c_k,\label{polk}
\end{equation}
\noindent which is the determinant of the shifted system derived from \req{SysAHO}.
The coefficients of $\tilde{p}_k(m)$ are
\begin{equation*}
a_k=-(\gamma_{11}^{(k)}+\gamma_{22}^{(k)}+\gamma_{33}^{(k)}),\label{ak}
\end{equation*}
\noindent 
\begin{equation*}
b_k=\gamma_{11}^{(k)}\gamma_{22}^{(k)}-(\gamma_{12}^{(k)})^2+\gamma_{11}^{(k)}\gamma_{33}^{(k)}-(\gamma_{13}^{(k)})^2+\gamma_{22}^{(k)}\gamma_{33}^{(k)}-(\gamma_{23}^{(k)})^2,\label{bk}
\end{equation*}
\noindent and
\begin{equation*}
c_k=-\gamma_{11}^{(k)}\gamma_{22}^{(k)}\gamma_{33}^{(k)}+\gamma_{11}^{(k)}(\gamma_{23}^{(k)})^2+\gamma_{22}^{(k)}(\gamma_{13}^{(k)})^2+\gamma_{33}^{(k)}(\gamma_{12}^{(k)})^2-2\gamma_{12}^{(k)}\gamma_{13}^{(k)}\gamma_{23}^{(k)},\label{ck}
\end{equation*}
\noindent with
\begin{equation*}
\gamma_{ii}^{(k)}=(\xi_i^{(k)})^2+(4\psi^2+3\omega^2+4\omega \psi)(\lambda_i^{(k)})^2+2(2\psi+\omega)\lambda_i^{(k)}\xi_i^{(k)}, \,\,\, i=1,2,3,\label{gammaii}
\end{equation*}
\noindent and
\begin{equation*}
\gamma_{ij}^{(k)}=\omega\left ((\lambda_i^{(k)}\xi_j^{(k)}+ \lambda_j^{(k)}\xi_i^{(k)})+(3\omega+4\psi)\lambda_i^{(k)}\lambda_j^{(k)} \right), \,\,\, i, j=1,2,3, \,\,\, i < j.\label{gammaij}
\end{equation*}

From the symmetry and positive definiteness of $\LL^{(k)}$ it follows that the polynomial (\ref{polk}) has three positive real roots $m_4^{(k)}$, $m_5^{(k)}$ and $m_6^{(k)}$, which are the sought eigenvalues of $\LL^{(k)}$. 

Taking into account the complementarity condition \req{gapk} 
and considering the four cases for the forms of
$\lambda_i^{(k)}$ and $\xi_i^{(k)}$ as done in the analysis of $\kappa(\SS_{NT}$, we get that in all cases $a_k$, $b_k$ and $c_k$ are polynomials of the variable $\mu_k$ of degree 2, 4 and 6, respectively.  Thus the roots of $\tilde{p}_k(m)$ have the expressions
\begin{equation}
 m_i^{(k)}=e_i^{(k)}+f_i^{(k)}\mu^{(k)}+g_i^{(k)}(\mu^{(k)})^2, \,\,\, i=4,5,6 \label{m456}
 \end{equation}
for some scalars $e_i^{(k)}$, $f_i^{(k)}$ and $g_i^{(k)}$. 

We point out that the eigenvalues of $\LL^{(k)}$ given in \req{m1}-\req{m3} are of the type \req{m456} with $e_i^{(k)}\neq0$ for $i=1,2,3$ in the four possible cases. Assuming $e_i^{(k)}\neq0$ for all $k$ and $i=4, 5, 6$, we can conclude that the condition number of $\MM_{AHO}^{(k)}$ does not depend on $\mu^{(k)}$. 
 We are aware that the assumption on $e_i^{(k)}$ in \req{m456}
is rather strong but we remark that it is fulfilled in all the performed experiments.

In addition, the independence of the condition number of $\MM_{AHO}^{(k)}$ on $\mu^{(k)}$ is corroborated by the analysis of the case $\CC=\mathbb{I}_{Sym}$. For this choice of $\CC$, tensor $\MM_{AHO}^{(k)}$ is symmetric and its eigenvalues are
\begin{equation*}
h_1^{(k)}=\frac{1}{2}(\xi_1^{(k)}+\xi_2^{(k)}+\lambda_1^{(k)}+\lambda_2^{(k)}), \quad
h_2^{(k)}=\frac{1}{2}(\xi_1^{(k)}+\xi_3^{(k)}+\lambda_1^{(k)}+\lambda_3^{(k)}), \label{e2}
\end{equation*}
\begin{equation*}
h_3^{(k)}=\frac{1}{2}(\xi_2^{(k)}+\xi_3^{(k)}+\lambda_2^{(k)}+\lambda_3^{(k)}), \label{e3}
\end{equation*}
with eigentensors in (\ref{q1q2}) and (\ref{q3}), and 
$$
h_4^{(k)}=\xi_1^{(k)}+\lambda_1^{(k)}, \quad
h_5^{(k)}=\xi_2^{(k)}+\lambda_2^{(k)}, \quad 
h_6^{(k)}=\xi_3^{(k)}+\lambda_3^{(k)}, 
$$
with eigentensors $\mathbf{q}_1^{(k)}\otimes\mathbf{q}_1^{(k)}$,
$\mathbf{q}_2^{(k)}\otimes\mathbf{q}_2^{(k)}$ and
$\mathbf{q}_3^{(k)}\otimes\mathbf{q}_3^{(k)}$, respectively.
Once again, taking into account the complementarity condition \req{gapk} 
and considering the four cases for the forms of
$\lambda_i^{(k)}$ and $\xi_i^{(k)}$, we get that $h_i^{(k)}$ have the expression in (\ref{m456}), with $g_i^{(k)}=0$ and $e_i^{(k)}\neq0$ and the desired result follows.

The calculation of the eigenvalues of $\LL^{(k)}$ for $\CC$ transversely isotropic is not easy and their explicit expressions are not available. Nevertheless, the numerical experiments reported in Section \ref{test} show that, as in the isotropic case, the condition number of $\MM_{AHO}^{(k)}$ does not depend on $\mu^{(k)}$.

\section{Numerical experiments} \label{test}
This section is devoted to numerical experiments; our purpose here is 
validating the proposed interior point approach for minimizing functional (\ref{functional}) and showing that it provides accurate solutions.
Moreover, we show that it is suitable for being implemented in the finite element code NOSA-ITACA \cite{nosaitaca} for the structural analysis of masonry constructions as it is able to solve problems where $\bD$ is the infinitesimal strain tensor calculated within NOSA-ITACA for each Gauss point (see the third data set in Section \ref{sec:dati}).

We first describe the details of the implemented methods and then discuss the 
testing sets and the numerical results. \\

We remark that in the following sections we focus on the problem formulation \req{sdp_primal}
in the description of both the algorithm and the experiments. Clearly, analogous considerations can be retrieved focusing on the minimization of (\ref{functional}) changing the variable $\bY=-\bX$.

%%%%%%%%%%%%%%%%%%%%%%%%%%%%%%%%%%%%%%%%%%%%%%%%%
\subsection{The \name algorithm and implementation details} \label{algo}

We report in Algorithm \ref{algo:pc} a pseudo-code for the \name method that is in fact an adaptation of the Mehrotra-type predictor corrector primal-dual algorithm \cite{QSDP2,QSDP1} applied to problem (\ref{sdp_primal}). This algorithm is very-well-known and is currently implemented in general purpose software for general QSDPs \cite{QSDP-guide}. It is a generalization of the method used in SDTP3 and cvx for linear semindefinite programming problems \cite{cvx-link,SDTP3-guide}.

In the description of the algorithm, let the current and the next iterate be $(\bX ,\bS)$ and
$(\bX ^+,\bS^+)$, respectively. Also, let the
current and the next step-length parameter be denoted by
$\tau$ and $\tau^+$, respectively.

\begin{algorithm}
\caption{\name}\label{algo:pc}
\begin{algorithmic}[1]
\Require Initial positive definite $(\bX ,\bS)$, accuracy level $\epsilon>0$.\\
Set {\tt flag = NT} if ${\mathbf{P}} = \bW^{-1}$ with $\bW$ in \req{eq:W} (NT direction); 
set {\tt flag = AHO} if ${\mathbf{P}}=\bI$ (AHO direction).
%Decide on the type of search direction to use (tensor $\mathbf{P}$). 
Set $\tau  = 0.9$.
 \Repeat %$\;k=1,2,\dots$}
 \State Set $\mu = \bX  \bullet \bS / 3$
 \State {\bf Convergence Test}
 \State Stop the iteration if the accuracy measure $\eta$ is sufficiently small, that is
 $$\eta\le \epsilon $$
  \State where
      $$
      \eta = \max \left \{ \frac{\bX  \bullet \bS}{1+ |p(\bX )| + |d(\bX )|}, 
      \frac{\|\bRd\|_F}{1+\|\CC(\bD)\|_F}  \right \},
      $$
\State $\bRd$ is defined in (\ref{Rd}) and the primal and dual objectives
$p(\bX )$ and $d(\bX )$ are
\State defined in (\ref{sdp_primal}) and (\ref{sdp_dual}), respectively. % and $pobj = \frac{1}{2} \bX  \bullet \CC(\bX ) + \bX  \bullet  \CC(E)$.

\State {\bf Predictor step}
\State Compute the predictor search direction $(\delta \bX , \delta \bS)$ by solving (\ref{New1})-(\ref{New2}) with $\sigma =0$ 
\State in the right hand-side $\bRc$ in (\ref{Rc}), in particular:
\If{ {\tt flag = NT}}
\State solve \req{sisX} and update $\delta \bX$  from \req{New1}
\Else{}% {\tt flag = AHO}}
\State solve \req{sisX1} and update $\delta \bS$ from \req{New1}
\EndIf
% qui serve dire che theta = Hp -sigma mu I
\State {\bf Predictor step-length}
\State Compute
 $$\alpha_p = \min (1, \tau \alpha)$$ \label{ap}
\State where $\alpha$ is the maximum step-length that can be taken so that $\bX +\alpha \delta \bX $ and
\State $\bS+\alpha \delta \bS$ remain positive semidefinite.
\State {\bf Centering rule}
\State Set $\sigma = (\bX +\alpha_p \delta \bX ) \bullet (\bS+\alpha_p \delta \bS)/\bX \bullet \bS$
\State {\bf Corrector step}
\State Compute the search direction
$(\Delta \bX , \Delta \bS)$ by solving (\ref{New1})-(\ref{New2}) with
$\bRc$ replaced by
$$
\bRc = \sigma \mu \bI - \HH_P(\bX \bS) -\HH_P(\delta \bX  \delta \bS),
$$
\State in particular:
\If{ {\tt flag = NT}}
\State solve \req{sisX} and update $\Delta \bX$  from \req{New1}
\Else{}% {\tt flag = AHO}}
\State solve \req{sisX1} and update $\Delta \bS$ from \req{New1}
\EndIf
\State {\bf Corrector step-length}
\State Compute
 $$\alpha_c = \min (1, \tau \alpha)$$  \label{ac}
\State where $\alpha$ is the maximum step-length that can be taken so that $\bX +\alpha \Delta \bX $ and
\State $\bS+\alpha \Delta \bS$ remain positive semidefinite.
\State {\bf Iterate and step-length update}
\State $\bX ^+ = \bX  + \alpha_c \Delta \bX ,\ \bS^+ = \bS + \alpha_c \Delta \bS$
\State $\tau^+ = 0.9 + 0.08 \alpha_c$.

\Until{convergence}
%\label{euclidendwhile}
\end{algorithmic}

\end{algorithm}
The step-length $\alpha$ is defined at Line \ref{ap} as
$\alpha = \min \{\alpha_\bX , \alpha_\bS \}$ with
$$ \alpha_\bX  = \left \{
\begin{array}{ll}
 \frac{-1}{\lambda_{\min}(\bX ^{-1}\delta \bX )} & \mbox{ if } \lambda_{\min}(\bX ^{-1}\delta \bX ) <0 \\
 \infty & \mbox{ otherwise},
\end{array}
\right .
$$
and
$$ \alpha_\bS = \left \{
\begin{array}{ll}
 \frac{-1}{\lambda_{\min}(\bS^{-1}\delta \bS)} & \mbox{ if } \lambda_{\min}(\bS^{-1}\delta \bS) <0 \\
 \infty & \mbox{ otherwise}.
\end{array}
\right .
$$
At Line  \ref{ac}, $\alpha$ is defined analogously replacing $\delta \bX $ and $\delta \bS$ with
$\Delta \bX $ and $\Delta \bS$, respectively.

We implemented \name in Matlab. In particular, regarding the NT direction, we closely followed the detailed implementation description in \cite{NT}.

The computation of the predictor and the corrector steps involves the computation and factorization either of the Schur complement $\SS_{NT}$ (if {\tt flag = NT}) or of the tensor $\MM_{AHO}$ (if {\tt flag = AHO}). We used the Cholesky factorization for $\SS_{NT}$  and the LU factorization with partial pivoting for $\MM_{AHO}$. Moreover, in computing $\alpha_{\bX}$ we computed the minimum eigenvalue of the symmetric tensor $\bL^{-1}\delta \bX \bL^{-T}$
where $\bX = \bL \bL^T$ is the Cholesky factorization of $\bX$; Analogously for 
$\alpha_{\bS}$ (in both Lines \ref{ap} and \ref{ac}).

In all the experiments, we used $(\bX, \bS ) = 2 (\bI, \bI)$
as a starting guess. Moreover, we set the accuracy level $\epsilon$ to the 
tight value $\epsilon = 10^{-15}$. Finally, a maximum number of 200 interior 
point iterations is allowed. The execution of the algorithm is also prematurely interrupted in case an error occurs in the  Cholesky factorization of $\bX$, $\bS$ or $\SS_{NT}$ meaning that these tensors are numerically loosing the positive definiteness.

Finally, we mention that \name has been implemented using the vector formalism  described in Section \ref{roba} as done in \cite{QSDP-guide,SDTP3-guide}.

%%%%%%%%%%%%%%%%%%%%%%%%%%%%%%%%%%%%%%%%%%%%%%%%%
\subsection{Data sets} \label{sec:dati}
The performance of the IPM-Proj algorithm is tested in on three data sets designed to highlight  the features of the projection problem and show the good behaviour of the algorithm for different choices of $\CC$. In what follows, we denote by $\CCiso$  the isotropic tensor defined in \req{Isotropy} and 
by $\CCtra$ the transversely isotropic tensor in \req{CTI}.

\paragraph{First data set: random $\bD$} 
We generated $10^5$ random symmetric tensors $\bD$ with random eigenvalues  in the interval $[-l,l]$ and  removed from the set tensors such that
$\bD \in Sym^-$ or such that $\CC(\bD) \in Sym^+$ as in these cases the solution is trivial.
Overall, we get sets of  74839, 70529 and 83373 tensors for $\CC = \mathbb{I}_{Sym}$, for $\CCiso$ and for $\CCtra$ with parameters given in Table \ref{tab:parC}, respectively. %stesso numero per l=1 e l=10
%Q = orth(rand(3));
% lambda = -lim + (lim+lim)*rand(3,1);

\paragraph{Second data set: Temple $\bD$} The goal of this data set is to test our algorithm with a view toward applications. To show the algorithm's efficiency without implementing it into the code NOSA-ITACA, we resort to an artificial case study constituted by the domed temple discretized into 31052 8-nodes hexahedral elements shown in Figure \ref{fig:massa}. First, the temple is subjected to its weight, and the strain field is calculated via a static analysis conducted with the NOSA-ITACA code.
\begin{figure}[h]
    \centering
    \includegraphics[height=.3\textheight]{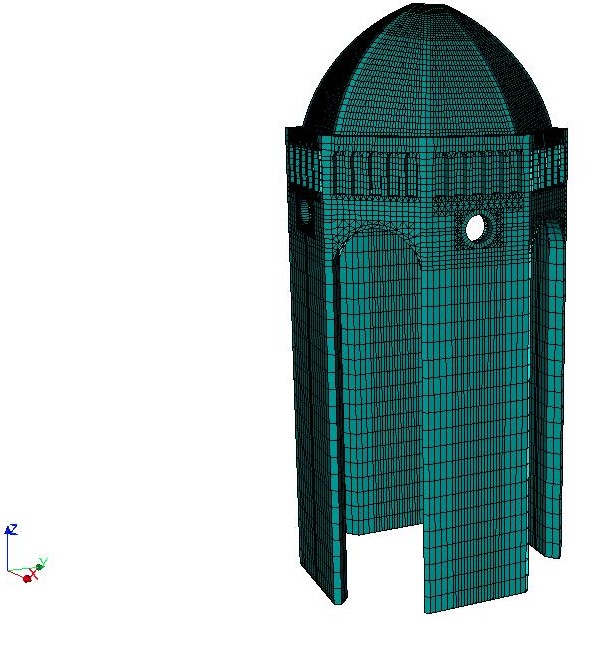}
    \caption{Finite element model of the domed temple.}
    \label{fig:massa}
\end{figure}
Then, tensors $\bD$ of the data set are the strain tensors calculated by NOSA-ITACA at each of 248216 Gauss points of the mesh. As for the random case, we removed from the set the tensors that gave trivial solutions.  
In particular, we tested tensors $\CCiso$ and for $\CCtra$ with 
values of $\nu$, $E$ and $\alpha$'s given in Table \ref{tab:parC} which are driven by physical considerations. Since in our tests the magnitude of these parameters in $\CC$ greatly differed from the the magnitude of the elements of $\bD$ generated by NOSA-ITACA, we found numerically convenient to normalize both $\bD$ and $\CC$, taking into account the properties of homogeneity \req{homog} and invariance \req{boh} of the projection introduced in Section \ref{sec2}. Taking into account the normalization, overall we get  89492 and 21747  tensors $\bD$ using $\CCiso$ and $\CCtra$, respectively.

%Then, the proposed algorithm has been run for each of the 248216 Gauss points of the mesh, using as tensor $\bD$ the strain calculated by NOSA-ITACA. Using parameters in Table \ref{tab:parC};

\paragraph{Third data set: parametric $\bD$} 
As pointed out in Section \ref{TIcase}, unlike the isotropic case, in the transversely isotropic case, the data $\bD$ and the exact solution $\bY^*$ (and then $\bX^*$) are in general not coaxial, and this lack of coaxiality seems to affect the performance of the \name algorithm.
Thus, the third data set is aimed at showing that the AHO direction behaves better than the NT direction when $\bD$ is not coaxial with $\bS$ and $\bX$. To this purpose, let  $(\mathbf{f}_{1},\mathbf{f}_{2},\mathbf{f}_{3})$ be an  orthonormal basis of  $\mathcal{V}$, with $\mathbf{f}_3$ the direction of transverse isotropy and let  us consider the orthonormal vectors
\begin{equation*}
\mathbf{q}_1(\theta_1, \theta_2)=-cos\theta_1 cos\theta_2 \mathbf{f}_1-sin\theta_1cos\theta_2 \mathbf{f}_2 + sin\theta_2 \mathbf{f}_3,
\end{equation*}
\begin{equation*}
\mathbf{q}_2(\theta_1, \theta_2)=sin\theta_1 \mathbf{f}_1-cos\theta_1 \mathbf{f}_2,
\end{equation*}
\begin{equation*}
\mathbf{q}_3(\theta_1, \theta_2)=cos\theta_1 sin\theta_2 \mathbf{f}_1 +sin\theta_1 sin\theta_2 \mathbf{f}_2 + cos\theta_2 \mathbf{f}_3,
\end{equation*}
and the tensors
\begin{equation*}
\mathbf{Q}_{ii}(\theta_1, \theta_2)=\mathbf{q}_{i}(\theta_1, \theta_2)\otimes \mathbf{q}_{i}(\theta_1, \theta_2), \,\,\, i=1,2,3,
\end{equation*}
\noindent for $\theta_1$ and  $\theta_2 \in [0,2\pi]$.
We then construct tensors $\mathbf{D}\in Sym$ of the type
\begin{equation}\label{Dnonc}
\mathbf{D}(\theta_1, \theta_2)=d_1\mathbf{Q}_{11}+d_2\mathbf{Q}_{22}+d_3\mathbf{Q}_{33}, \text{  with  }d_1, d_2, d_3 \in \mathbb{R}.
\end{equation}
For $\theta_2=0, \pi, 2\pi$  we have $\mathbf{q}_3(\theta_1, 0)=\mathbf{f}_3$, thus tensors $\bD$ and the solution $\bY^*$ are coaxial and the explicit solution is provided in (\ref{TI0})-(\ref{TI3}); for any other choice of $\theta_2$, $\bD$ and $\bY^*$ deviates from coaxiality.

\begin{table}[htbp]
  \centering   
  \caption{Parameters used in the experiments for the definition of the tensors $\CC$.}
    \begin{tabular}{l|rr|rrrrr}
    \toprule
& \multicolumn{2}{c}{$\CCiso$} & \multicolumn{5}{|c}{$\CCtra$} \\
\midrule
      &   \multicolumn{1}{c}{$E$} & \multicolumn{1}{c|}{$\nu$} & \multicolumn{1}{c}{$\alpha_1$} & \multicolumn{1}{c}{$\alpha_2$} & \multicolumn{1}{c}{$\alpha_3$}& \multicolumn{1}{c}{$\alpha_4$} & \multicolumn{1}{c}{$\alpha_5$}\ \\
      \midrule
 Random $\bD$     & 1 & 0.1 & 8 & 2 & 0.8 & 6 & 9 \\
 Temple $\bD$ & $3\cdot 10^9$ & 0.2 &  $2.0338\cdot 10^9$ & $ 1.91 \cdot 10^9$ & $ 2.5423 \cdot 10^8$ & $2\cdot 10^9$ & $ 1.25\cdot 10^9 $ \\
Parametric $\bD$     & 1 & 0.1 & 8 & 2 & 0.8 & 6 & 9 \\
    \bottomrule
    \end{tabular}%

  \label{tab:parC}%
\end{table}%

\subsection{Numerical results} \label{sec:risu}

All results given in this section were obtained on an
a Intel Core i7-9700K PC running at 3.60 GHz x 8 with 16 GB of RAM, 64-bit
and using Matlab R2019b.

We first discuss the performance of the proposed algorithm and the effectiveness of the symmetrization schemes described in Section \ref{numeric} on the first two data sets. As a measure of performance we use the complementarity gap defined as $gap = \bX\bullet \bS$ and when an analytic solution $\bX^*$ is available, i.e. when $\CC = \mathbb{I}_{Sym}$ or with $\CCiso$, the absolute error computed as $error = \|\bX-\bX^*\|$. 
We observe that the $gap$ can be interpreted as coaxiality measure, being 
$\bX^*\bullet \bS^* = \bY^* \bullet \CC(\bD-\bY^*)$.

We report in Table \ref{tab:risumedi} the average number of interior point iterations $iter_{Av}$, the average complementarity gap $gap_{Av}$, the average absolute error $error_{Av}$ and the total CPU time $cpu_{tot}$. The symbol `-' means that
that $error_{Av}$ is not available (transversely isotropic case).
We observe that both the accuracy measures $error$ and $gap$ are
in favour of the \name implementing the AHO direction as, in fact, using the NT
direction yields from 3 to 6 less order of accuracy. 
Moreover, the use of the NT direction implies, on average, a larger number of 
iterations when $\CCtra$ is used. Both issues are related to 
the fact that the Schur complement $\SS_{NT}$ becomes very-ill conditioned for small $\mu$
yielding poor interior point directions and, in several cases, runs are  prematurely stopped due to an error in the Cholesky factorization.  
Conversely, the nice condition number of $\MM_{AHO}$ as $\mu \rightarrow 0$ allows to compute  very accurate solutions. Finally, although aware that
evaluating the cpu time of Matlab codes is not always meaningful, especially when built-in functions are employed, we note that \name with the AHO direction is faster than with the NT one.

\begin{table}[htbp]
  \centering
  \caption{Aggregated results for the experiments on random $\bD$
  (for $l=1,10$) and Temple $\bD$ varying $\CC$: average number of iterations $iter_{Av}$, average 
  complementarity gap $gap_{Av}$, average absolute error $error_{Av}$ and the total CPU time $cpu_{tot}$.}
    \begin{tabular}{ll|rrrrrrrr}
    \toprule
          &       & \multicolumn{2}{c}{$iter_{Av}$} & \multicolumn{2}{c}{$gap_{Av}$} & \multicolumn{2}{c}{$error_{Av}$} & \multicolumn{2}{c}{$cpu_{tot}$ } \\
          \midrule
  random $\bD$        &  & \multicolumn{1}{c}{AHO} & \multicolumn{1}{c}{NT} & \multicolumn{1}{c}{AHO} & \multicolumn{1}{c}{NT} & \multicolumn{1}{c}{AHO}   & \multicolumn{1}{c}{NT}    & \multicolumn{1}{c}{AHO} & \multicolumn{1}{c}{NT} \\
          \midrule
   $\CC = \mathbb{I}_{Sym}$ & $l=1$     & 12    & 11    & 3E-16 & 1E-10 & \multicolumn{1}{r}{6E-15} & \multicolumn{1}{r}{6E-10} & 68.1  & 89.7 \\
          & $l=10$    & 11    & 10    & 1E-14 & 1E-08 & \multicolumn{1}{r}{1E-14} & \multicolumn{1}{r}{1E-08} & 63.5  & 91.6 \\
                    \midrule
    $\CCiso$ & $l=1$     & 12    & 11    & 3E-16 & 1E-10 & \multicolumn{1}{r}{5E-15} & \multicolumn{1}{r}{7E-10} & 64.2  & 84.9 \\  %(E=1.nu=0.1)
          & $l=10$    & 11    & 10    & 1E-14 & 1E-08 & \multicolumn{1}{r}{2E-14} & \multicolumn{1}{r}{6E-09} & 59.3  & 87.7 \\
                    \midrule
    $\CCtra$  & $l=1$     & 12    & 20    & 2E-15 & 8E-08 & -     & -    & 73.3  & 178 \\  %alpha = 8, 2, 0.8, 6, 9
          & $l=10$    & 12    & 22    & 7E-13 & 8E-06 & -     & -     & 78.8  & 199.4 \\
                    \bottomrule
         \multicolumn{10}{c}{} \\
    \toprule         
          &       & \multicolumn{2}{c}{$iter_{Av}$} & \multicolumn{2}{c}{$gap_{Av}$} & \multicolumn{2}{c}{$error_{Av}$} & \multicolumn{2}{c}{$cpu_{tot}$ } \\
          \midrule
      Temple $\bD$ &  & \multicolumn{1}{c}{AHO} & \multicolumn{1}{c}{NT} & \multicolumn{1}{c}{AHO} & \multicolumn{1}{c}{NT} & \multicolumn{1}{c}{AHO}   & \multicolumn{1}{c}{NT}    & \multicolumn{1}{c}{AHO} & \multicolumn{1}{c}{NT} \\
          \midrule

    $\CCiso$   &       & 11    & 11    & 8E-13 & 3E-11 & \multicolumn{1}{r}{2E-11} & \multicolumn{1}{r}{2E-10} & 78.9 & 110.6 \\
        $\CCtra$  &       & 11    & 17    & 4E-13 & 1E-10 & -     & -     & 19.5  & 41.9\\
    \bottomrule
    \end{tabular}%
  \label{tab:risumedi}%
\end{table}%

In order to deepen the analysis on the condition number of $\SS_{NT}$
and $\MM_{AHO}$, we randomly drew one tensor $\bD$ from the random data set and
one from the temple one. For both tensors, we plotted in Figures \ref{fig:cond_trans} and \ref{fig:cond_t_massa} the values of $\kappa(\SS_{NT})$
and $\kappa(\MM_{AHO})$ along the interior point (IPM) iterations, for both $\CCiso$ and $\CCtra$, together with the value of $1/\mu$ for a matter of comparison.
As expected $\kappa(\SS_{NT})$ grows as $1/\mu$, while $\kappa(\MM_{AHO})$ is 
constant both for $\CCiso$ and $\CCtra$ as discussed in Section \ref{condSM}.

%%%%%%%%%%%%%%%%%%%%%%%%%%%%%%%%%%%%%%%%%%%%%%%%%%%%%%%%%%%%%%%%%%%%%%%%%%%%%%%%
% COND RANDOM
%%%%%%%%%%%%%%%%%%%%%%%%%%%%%%%%%%%%%%%%%%%%%%%%%%%%%%%%%%%%%%%%%%%%%%%%%%%%%%%%%
	
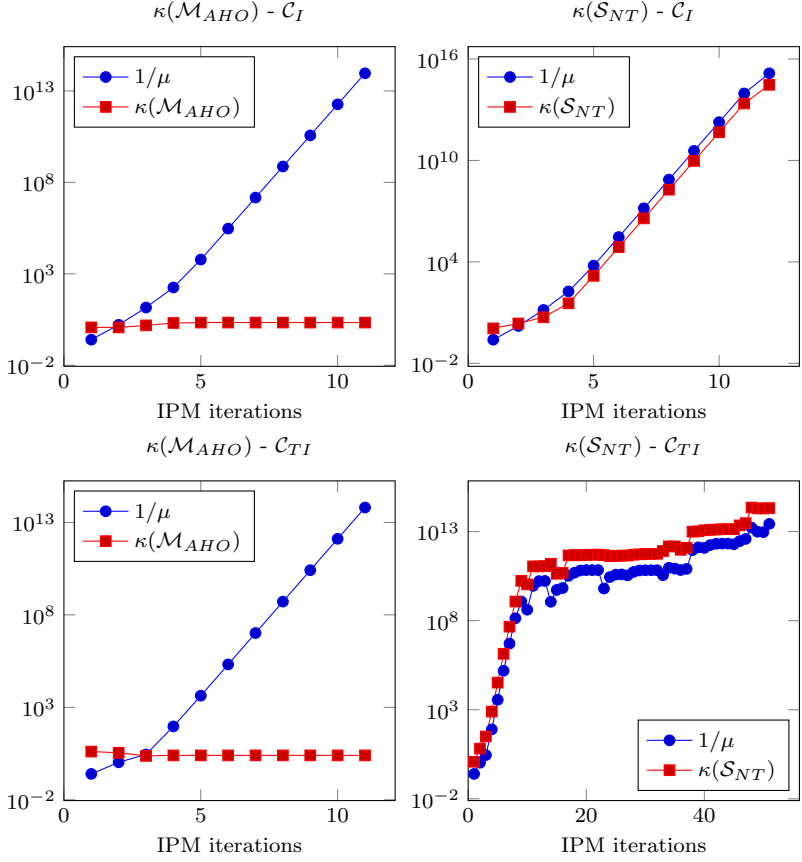
\begin{figure} \centering
	\begin{tikzpicture}
		\begin{semilogyaxis}[width=. 5* \textwidth, 
		   title = {$\kappa(\MM_{AHO})$ - $\CC_{I}$}, 
		   legend pos = north west,
		   legend cell align=left,
		   xlabel = {IPM iterations}, 
		   % ylabel + [Y LABEL],
		   height = .3 \textheight,
		   xmin = 0,
		   		 ]
		\foreach \j in {1,3} {
			% Mark = none per togliere i pallozzi
		  \addplot+ table[x index = 0, y index = \j] {./dati_randD_ii1_iso_M_AHO.dat};
		}
		\legend{ $1/\mu$,  $\kappa(\MM_{AHO})$};
		\end{semilogyaxis}
			\end{tikzpicture}~\begin{tikzpicture}
			\begin{semilogyaxis}[width=. 5* \textwidth, 
		   title = { $\kappa(\SS_{NT})$ - $\CC_{I}$}, 
		   legend pos = north west,
		   legend cell align=left,
		   xlabel = {IPM iterations}, 
		   % ylabel + [Y LABEL],
		   height = .3 \textheight,
		   xmin = 0,
		   		 ]
		\foreach \j in {1,3} {
			% Mark = none per togliere i pallozzi
		  \addplot+ table[x index = 0, y index = \j] {./dati_randD_ii1_iso_H_NT.dat};
		}
		\legend{ $1/\mu$,   $\kappa(\SS_{NT})$};
		\end{semilogyaxis}
	\end{tikzpicture}\\
	
	\begin{tikzpicture}
		\begin{semilogyaxis}[width=. 5* \textwidth, 
		   title = {$\kappa(\MM_{AHO})$  - $\CCtra$}, 
		   legend pos = north west,
		   legend cell align=left,
		   xlabel = {IPM iterations}, 
		   % ylabel + [Y LABEL],
		   height = .3 \textheight,
		   xmin = 0,
		   		 ]
		\foreach \j in {1,3} {
			% Mark = none per togliere i pallozzi
		  \addplot+ table[x index = 0, y index = \j] {./dati_randD_ii1_trans_M_AHO.dat};
		}
		\legend{ $1/\mu$,  $\kappa(\MM_{AHO})$};
		\end{semilogyaxis}
			\end{tikzpicture}~\begin{tikzpicture}
			\begin{semilogyaxis}[width=. 5* \textwidth, 
		   title = {$\kappa(\SS_{NT})$ - $\CCtra$}, 
		   legend pos = south east,
		   legend cell align=left,
		   xlabel = {IPM iterations}, 
		   % ylabel + [Y LABEL],
		   height = .3 \textheight,
		   xmin = 0,
		   		 ]
		\foreach \j in {1,3} {
			% Mark = none per togliere i pallozzi
		  \addplot+ table[x index = 0, y index = \j] {./dati_randD_ii1_trans_H_NT.dat};
		}
		\legend{ $1/\mu$,   $\kappa(\SS_{NT})$};
		\end{semilogyaxis}
	\end{tikzpicture}
	\caption{Random $\bD$: condition number of $\SS_{NT}$ and $\MM_{AHO}$ along the interior point (IPM) iterations.}  \label{fig:cond_trans} %prima matrice
	\end{figure}	

%%%%%%%%%%%%%%%%%%%%%%%%%%%%%%%%%%%%%%%%%%%%%%%%%%%%%%%%%%%%%%%%%%%%%%%%%%%%%%%%%
% COND MASSA
%%%%%%%%%%%%%%%%%%%%%%%%%%%%%%%%%%%%%%%%%%%%%%%%%%%%%%%%%%%%%%%%%%%%%%%%%%%%%%%%%
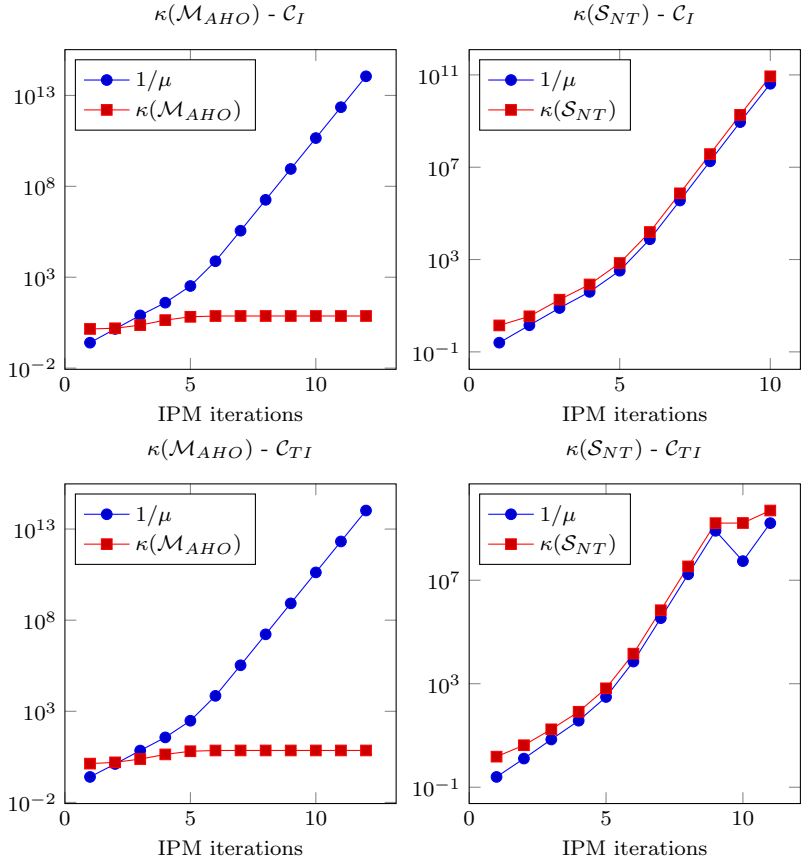
\begin{figure} \centering
	\begin{tikzpicture}
		\begin{semilogyaxis}[width=. 5* \textwidth, 
		   title = {$\kappa(\MM_{AHO})$ - $\CC_{I}$}, 
		   legend pos = north west,
		   legend cell align=left,
		   xlabel = {IPM iterations}, 
		   % ylabel + [Y LABEL],
		   height = .3 \textheight,
		   xmin = 0,
		   		 ]
		\foreach \j in {1,3} {
			% Mark = none per togliere i pallozzi
		  \addplot+ table[x index = 0, y index = \j] {./dati_massa_ii170133_iso_M_AHO.dat};
		}
		\legend{ $1/\mu$,  $\kappa(\MM_{AHO})$};
		\end{semilogyaxis}
			\end{tikzpicture}~\begin{tikzpicture}
			\begin{semilogyaxis}[width=. 5* \textwidth, 
		   title = { $\kappa(\SS_{NT})$ - $\CC_{I}$}, 
		   legend pos = north west,
		   legend cell align=left,
		   xlabel = {IPM iterations}, 
		   % ylabel + [Y LABEL],
		   height = .3 \textheight,
		   xmin = 0,
		   		 ]
		\foreach \j in {1,3} {
			% Mark = none per togliere i pallozzi
		  \addplot+ table[x index = 0, y index = \j] {./dati_massa_ii170133_iso_H_NT.dat};
		}
		\legend{ $1/\mu$, $\kappa(\SS_{NT})$};
		\end{semilogyaxis}
	\end{tikzpicture}\\
	
	\begin{tikzpicture}
		\begin{semilogyaxis}[width=. 5* \textwidth, 
		   title = {$\kappa(\MM_{AHO})$  - $\CCtra$}, 
		   legend pos = north west,
		   legend cell align=left,
		   xlabel = {IPM iterations}, 
		   % ylabel + [Y LABEL],
		   height = .3 \textheight,
		   xmin = 0,
		   		 ]
		\foreach \j in {1,3} {
			% Mark = none per togliere i pallozzi
		  \addplot+ table[x index = 0, y index = \j] {./dati_massa_ii170133_trans_M_AHO.dat};
		}
		%\legend{ $1/\mu$, $1/\sqrt{\mu}$, $\kappa(\MM_{AHO})$};
		\legend{ $1/\mu$, $\kappa(\MM_{AHO})$};
		\end{semilogyaxis}
			\end{tikzpicture}~\begin{tikzpicture}
			\begin{semilogyaxis}[width=. 5* \textwidth, 
		   title = {$\kappa(\SS_{NT})$ - $\CCtra$}, 
	   legend pos = north west,
		   legend cell align=left,
		   xlabel = {IPM iterations}, 
		   % ylabel + [Y LABEL],
		   height = .3 \textheight,
		   xmin = 0,
		   		 ]
		\foreach \j in {1,3} {
			% Mark = none per togliere i pallozzi
		  \addplot+ table[x index = 0, y index = \j] {./dati_massa_ii170133_trans_H_NT.dat};
		}
		\legend{ $1/\mu$,   $\kappa(\SS_{NT})$};
		%\legend{ $1/\mu$, $1/\sqrt{\mu}$,  $\kappa(\SS_{NT})$};
		\end{semilogyaxis}
	\end{tikzpicture}
	\caption{Temple $\bD$: condition number of $\SS_{NT}$ and $\MM_{AHO}$ along the interior point (IPM) iterations.} \label{fig:cond_t_massa} %(scelta random ii = 170133)
	\end{figure}

In order to interpret the results discussed above with a further tool,
we report in Figure \ref{fig:massabp} the boxplots of the runs performed for 
the Temple $\bD$ related to the $\log_{10}(gap)$ computed using the AHO and the NT directions for both  $\CCiso$ and $\CCtra$. 
The plots show that the maximum, i.e. the highest data point in the data set excluding any outliers, is larger using NT than using AHO. Moreover, the use of NT yields a large number of outliers with values above the maximum. 
We remark that these outliers correspond to runs prematurely stopped for a failure in the Cholesky factorization due to the ill-conditioning of the Schur complement $\SS_{NT}$.

%%%%%%%%%%%%%%%%%%%%%%%%%%%%%%%%%%%%%%%%%%%%%%%%%%%%%%%%%%%%%%%
	%BOX PLOT
\begin{figure}
\centering
%\subfloat[ $\log_{10}(error)$ - $\CCiso$]{
	  %\includegraphics[width=.45\textwidth, height=.6\textheight, %keepaspectratio]{massa_Aerr_iso_norm2.eps}
	  %}
	  \subfloat[ $\log_{10}(gap)$ - $\CCiso$]{
	  \includegraphics[width=.45\textwidth, height=.6\textheight, keepaspectratio]{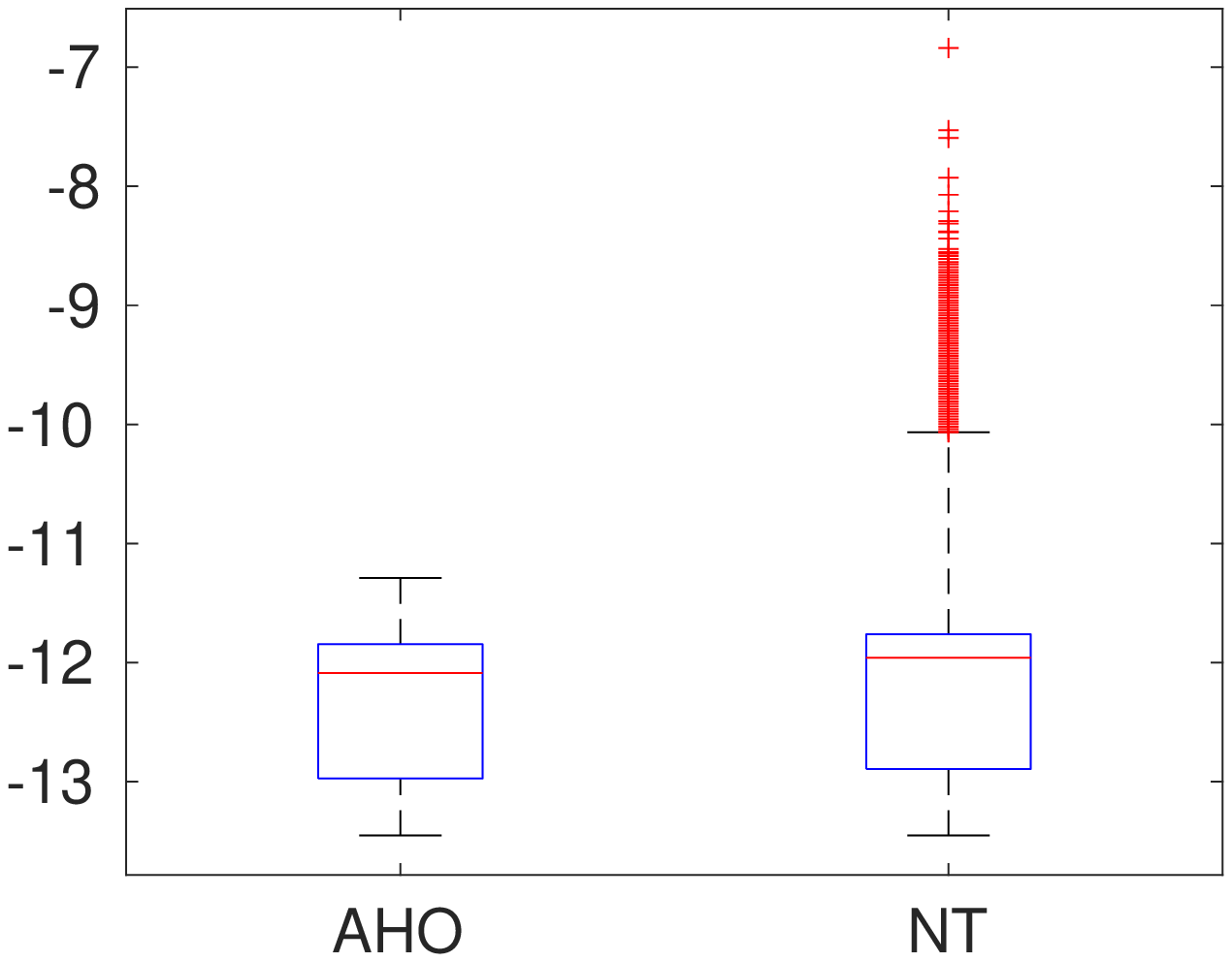}
	  }
	   \subfloat[ $\log_{10}(gap)$  - $\CCtra$]{
	  	  	  \includegraphics[width=.45\textwidth, height=.6\textheight, keepaspectratio]{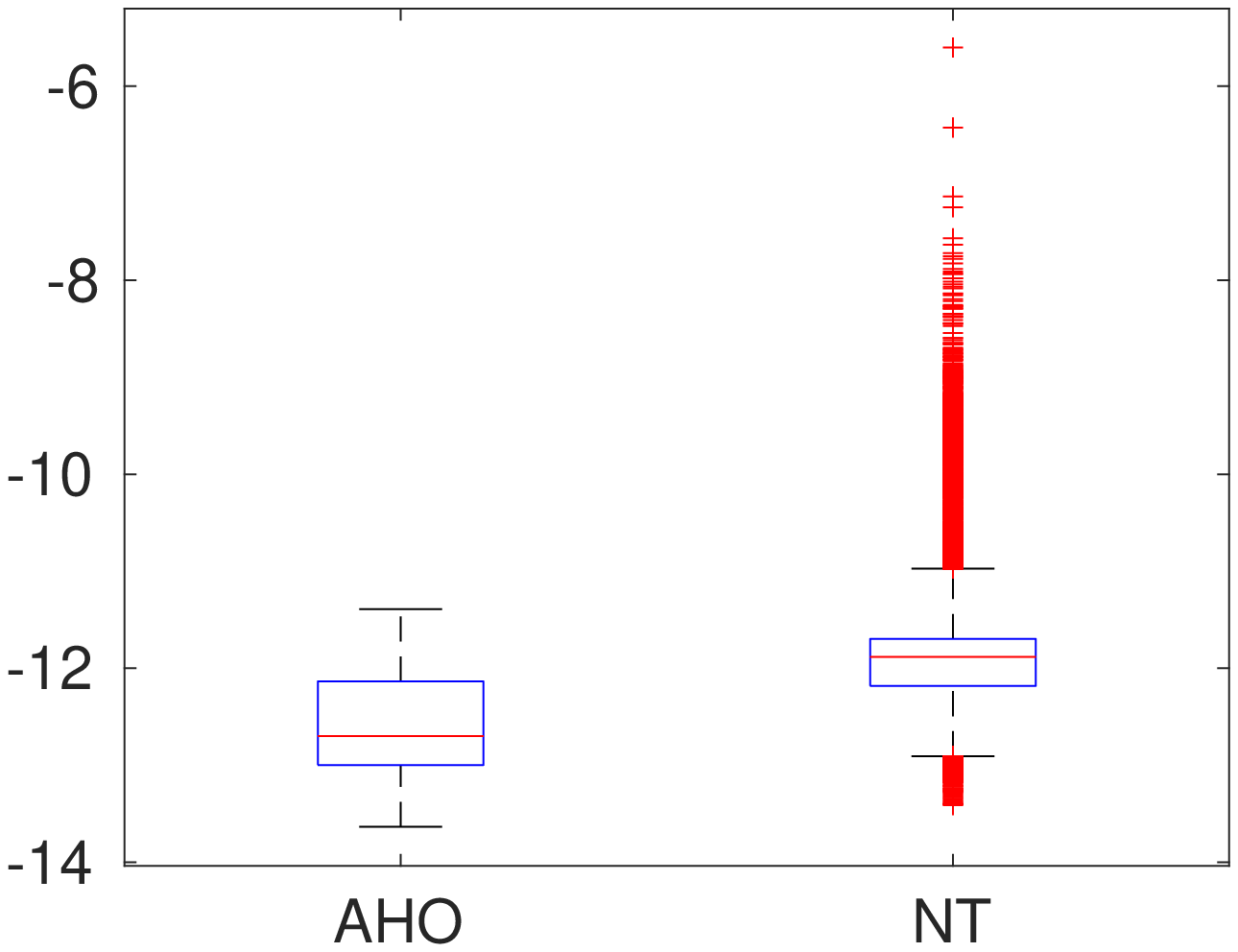}
	  	  	  }
\caption{Temple $\bD$: box-plots of the $\log_{10}$ of  the complementarity gap  $ gap = X\bullet S$ using $\CCiso$ (plot (a)) and using $\CCtra$ (plot (b)).
}
\label{fig:massabp}
\end{figure}

Concerning the third data set, Figure \ref{fig:evf3} reports the plot of the complementarity gap versus the angle $\theta_2$, for $\theta_1=\pi/4$,
and $(d_1,d_2,d_3)$ randomly chosen in $[-1,1]$, when $\CCtra$ is employed. The trend of the complementarity gap clearly shows how the deviation from coaxiality of $\bD$ and $\bX$ influences the coaxiality of $\bX$ and $\bS$ and then the accuracy of the numerical solution.
The use of the AHO direction seems to mitigate this effect.
The solutions calculated using the AHO and NT direction coincide for $\theta_2=0, \pi, 2\pi$.
 
 %%%%%%%%%%%%%%%%%%%%%%%%%%%%%%%
% f3
	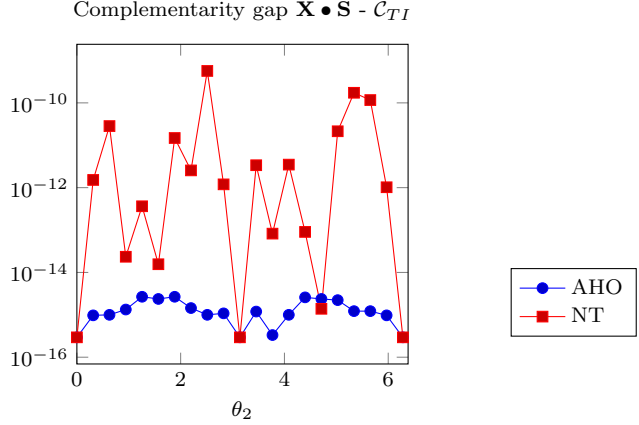
\begin{figure} \centering
	\begin{tikzpicture}
		\begin{semilogyaxis}[width=. 5* \textwidth, 
		   title = {Complementarity gap $\bX\bullet \bS$ - $\CCtra$}, 
		   legend style={at={(1.5,0.3)},anchor=north},
		   legend cell align=left,
		   xlabel = {$\theta_2$}, 
		   % ylabel + [Y LABEL],
		   height = .3 \textheight,
			 xmax = 2*pi+0.1,
		   xmin = 0,
		   		 ]
		\foreach \j in {1,2} {
			% Mark = none per togliere i pallozzi
		  \addplot+ table[x index = 0, y index = \j] {./dati_evf3.dat};
		}
		\legend{ AHO, NT  };
	 %    \legend{ CG iters in IPLR-GS for (\ref{sis2}), CG iters in IPLR-GS-P for (\ref{sis2})};
		\end{semilogyaxis}
	\end{tikzpicture}
% 	%  \includegraphics[width=0.6\textwidth]{./vbeam.eps}
	\caption{Parametric $\bD$: value of the complementarity gap at the computed solution varying $\theta_2$ in the definition of $\bD$ in \req{Dnonc}.} \label{fig:evf3}
	\end{figure}

\section{Conclusions} \label{sec:end} 

In this paper, we addressed a projection problem consisting in determining the projection of a symmetric second-order tensor onto the cone of negative semidefinite symmetric tensors with respect to the inner product defined by an assigned positive definite symmetric fourth-order tensor $\CC$. 
Applications of interests in solid mechanics strongly motivated this work 
supplying special forms for the tensors $\CC$ which require the numerical
solution of the projection problem. To this purpose, we considered an interior 
point method for a semidefinite programming reformulation of the problem
and discuss reliable implementations based on direct solvers for the linear algebra. Several numerical tests are performed to validate the proposed method
showing that the use of the AHO direction might be preferable to get accurate solutions.

The implementation of the algorithm in the finite element code NOSA-ITACA \cite{nosaitaca} developed at ISTI-CNR for the structural analysis
of masonry constructions will be the subject of future work
together with the analysis of a real-world case study of engineering interest.

\begin{acknowledgements}
The second author is a member of the {\em Gruppo Nazionale per il Calcolo Scientifico} (GNCS) of the Istituto Nazionale di Alta Matematica (INdAM) and this work was partially supported by INdAM-GNCS under Progetti di Ricerca 2020. 
\end{acknowledgements}

\begin{comment}
\section*{Declarations}

Some journals require declarations to be submitted in a standardised format. Please check the Instructions for Authors of the journal to which you are submitting to see if you need to complete this section. If yes, your manuscript must contain the following sections under the heading `Declarations':

\begin{itemize}
\item Funding
\item Conflict of interest/Competing interests (check journal-specific guidelines for which heading to use)
\item Ethics approval 
\item Consent to participate
\item Consent for publication
\item Availability of data and materials
\item Code availability 
\item Authors' contributions
\end{itemize}
\noindent
If any of the sections are not relevant to your manuscript, please include the heading and write `Not applicable' for that section. 
\end{comment}

\appendix

\section{Components of tensor $\CC$}
The components $\CC_{ijkl}$ of $\CC$ and $\CC_{ijkl}^{-1}$ of $\CC^{-1}$ with respect to an orthonormal basis $\mathsf{P}= (\mathbf{p}_{1},\mathbf{p}_{2},\mathbf{p}_{3})$ of $\mathcal{V}$ are introduced in Section \ref{sec1}. These components are reported in the following for the fourth-order tensors $\CC$ used in the numerical experiments.

In the isotropic case we have
\begin{equation*}
\CC_{1111}=\CC_{2222}=\CC_{3333}=\frac{E(1-\nu)}{(1+\nu)(1-2\nu)},
\label{Cijkliso1}
\end{equation*}
\begin{equation*}
\CC_{1122}=\CC_{1133}=\CC_{2233}= \frac{E\nu}{(1+\nu)(1-2\nu)},
\label{Cijkliso2}
\end{equation*}
\begin{equation*}
\CC_{1212}=\CC_{1313}=\CC_{2323}= \frac{E}{2(1+\nu)},
\label{Cijkliso3}
\end{equation*}
\begin{equation*}
\CC_{1111}^{-1}=\CC_{2222}^{-1}=\CC_{3333}^{-1}=\frac{1}{E},\ \ 
\CC_{1122}^{-1}=\CC_{1133}^{-1}=\CC_{2233}^{-1}= -\frac{\nu}{E},
\ \
\CC_{1212}^{-1}=\CC_{1313}^{-1}=\CC_{2323}^{-1}= \frac{1+\nu}{2E}.
\label{Cijkliso3-1}
\end{equation*}
The other components, if not zero, are given by relations (\ref{simmetrieC}) and (\ref{simmetrieC-1}).

In the transversely isotropic case, if $\mathbf{p}_{3}$ is the direction of transverse isotropy, then we have
\begin{equation*}
\CC_{1111}=\CC_{2222}=\alpha_2 + \alpha_5,\textup{ \ \ } \CC_{3333}=\alpha_1,
\label{CijklTI1}
\end{equation*}
\begin{equation*}
\CC_{1122}=\alpha_2-\alpha_5, \textup{ \ \ } \CC_{1133}=\CC_{2233}= \alpha_3,
\label{CijklTI2}
\end{equation*}
\begin{equation*}
\CC_{1212}=\alpha_5, \textup{ \ \ } \CC_{1313}=\CC_{2323}= \alpha_4,
\label{CijklTI3}
\end{equation*}
\begin{equation*}
\CC_{1111}^{-1}=\CC_{2222}^{-1}=\frac{ \alpha_1(\alpha_2+\alpha_5)-\alpha_3^2}{\delta},\textup{ \ \ }
\CC_{3333}^{-1}=\frac{ 4 \alpha_2 \alpha_5}{\delta},
\label{CijklTI-11}
\end{equation*}
\begin{equation*}
\CC_{1122}^{-1}=\frac{\alpha_3^2- \alpha_1(\alpha_2-\alpha_5)}{\delta},\textup{ \ \ }
\CC_{1133}^{-1}=\CC_{2233}^{-1}= \frac{-2 \alpha_3 \alpha_5}{\delta},
\label{CijklTI-12}
\end{equation*}
\begin{equation*}
\CC_{1212}^{-1}=\frac{ 1}{ 4\alpha_5} , \textup{ \ \ } \CC_{1313}^{-1}=\CC_{2323}^{-1}=\frac{ 1}{ 4\alpha_4},
\label{CijklTI-13}
\end{equation*}
with
\begin{equation*}
\delta=4\alpha_5(\alpha_1 \alpha_2 -\alpha_3^2),
\label{Det}
\end{equation*}
and the remaining components are defined by (\ref{simmetrieC}) and (\ref{simmetrieC-1}) or are equal to zero.

%%%%%%%%%%%%%%%%%%%%% COMMENTATO
\begin{comment}

%Different choices of $P$ gives different directions., e.g.:
\begin{itemize}
 \item {\bf AHO}: ${\mathbf{P}} = \bI$:
 $$\EE = \bI \odot \bS, \quad \FF = \bX  \odot \bI,\quad \bRc= \mu \bI -\frac{1}{2}(\bX \bS+\bS\bX ).
 $$
 $$\SS = (\bI \odot \bS)^{-1} (\bX  \odot \bI) + \CC^{-1}
 $$
 and
  $$\MM = \bI \odot \bS + (\bX  \odot \bI)\CC
 $$

 \item {\bf NT}: ${\mathbf{P}} = \bW^{-1/2}$ with $\bW$ being the geometric mean of $\bX $ and $\bS^{-1}$, i.e.
                 $$\bW= \bX ^{1/2}(\bX ^{1/2}\bS \bX ^{1/2})^{-1/2} \bX ^{1/2} =
                 \bS^{-1/2}(\bS^{1/2}\bX  \bS^{1/2})^{1/2} \bS^{-1/2}$$
                 ($\bW\bS\bW=\bX $).

%                This gives
 $$\EE = \bS \odot \bW^{-1}, \quad \FF = \bW \odot \bW,\quad \bRc= \mu \bS^{-1} - \bX .
 $$

  $$\SS = (\bS^{-1}\bW \odot \bI) + \CC^{-1}
 $$
 or
  $$\MM =\bS \odot \bW^{-1} + (\bW \odot \bW) \CC
 $$
Let $\bW = \bG \bG^T$, ${\mathbf{P}} = \bG^{-1}$

 $$\EE = \bG^{-1} \odot \bG^T\bS, \quad \FF = \bG^{-1}\bX \odot \bG^T,\quad \bRc= \sigma \mu \bI  - {\mathbf{D}}.
 $$
  $$\SS = (\bW \odot \bW) + \CC^{-1}
 $$
 and
  $$\MM = \bG^{-1} \odot \bG^T\bS + (\bG^{-1}\bX \odot \bG^T) \CC
 $$

\end{itemize}

\end{comment}

\def\baselinestretch{1}

\end{document}